\documentclass[12pt,twoside]{amsart}
\usepackage[margin=3cm]{geometry}
\usepackage[colorlinks=false]{hyperref}
\usepackage[english]{babel}
\usepackage{graphicx,titling}
\usepackage{float}
\usepackage{amsmath,amsfonts,amssymb,amsthm}
\usepackage{lipsum}
\usepackage[T1]{fontenc}
\usepackage{fourier}
\usepackage{color}
\usepackage[latin1]{inputenc}
\usepackage{esint}
\usepackage{caption}
\usepackage{ccicons}

\makeatletter
\def\blfootnote{\xdef\@thefnmark{}\@footnotetext}
\makeatother

\newcommand\ccnote{
    \blfootnote{\copyright\,\, Nets Hawk Katz, Shukun Wu, and Joshua Zahl}
    \blfootnote{\ccLogo\, \ccAttribution\,\, Licensed under a \href{https://creativecommons.org/licenses/by/4.0/}{Creative Commons Attribution License (CC-BY)}.}
}

\usepackage[export]{adjustbox}
\numberwithin{equation}{section}
\usepackage{setspace}

\renewcommand{\leq}{\leqslant}

\renewcommand{\geq}{\geqslant}
\renewcommand{\mathbb}{\varmathbb}
\usepackage{fancyhdr}
\usepackage[percent]{overpic}

\newcommand{\RR}{\mathbb{R}}
\newcommand{\CC}{\mathbb{C}}

\newcommand{\eps}{\varepsilon}
\newtheorem{thm}{Theorem}[section]
\newtheorem{lem}[thm]{Lemma}
\newtheorem{prop}[thm]{Proposition}

\newtheorem{conj}[thm]{Conjecture}

\theoremstyle{remark}
\newtheorem{defn}[thm]{Definition}

\newtheorem{rem}[thm]{Remark}

\address{Nets Hawk Katz, Rice University, Department of Mathematics, Houston, TX, USA}
\email{nets@rice.edu}
\medskip
\address{Shukun Wu, Indiana University Bloomington, Department of Mathematics, Bloomington, IN, USA} 
\email{shukwu@iu.edu}
\medskip
\address{Joshua Zahl, The University of British Columbia, Department of Mathematics, Vancouver, BC, Canada}
\email{jzahl@math.ubc.ca}

\begin{document}

\thispagestyle{empty}

\begin{minipage}{0.28\textwidth}
\begin{figure}[H]
\includegraphics[width=2.5cm,height=2.5cm,left]{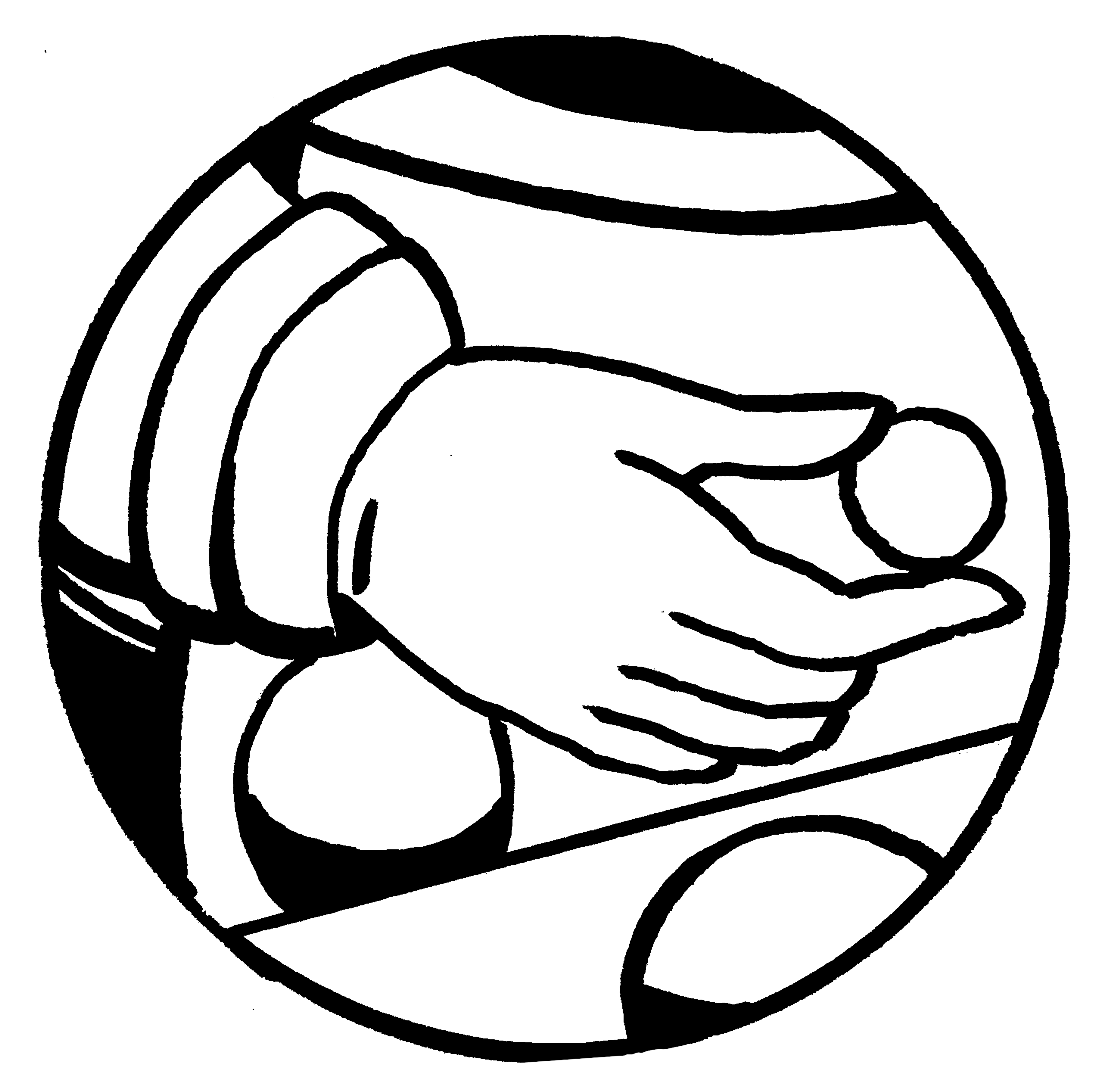}
\end{figure}
\end{minipage}
\begin{minipage}{0.7\textwidth} 
\begin{flushright}
Ars Inveniendi Analytica (2023), Paper No. 6, 23 pp.
\\
DOI 10.15781/ep89-2217
\\
ISSN: 2769-8505
\end{flushright}
\end{minipage}

\ccnote

\vspace{1cm}


\begin{center}
\begin{huge}
\textit{Kakeya sets from lines in $SL_2$}


\end{huge}
\end{center}

\vspace{1cm}


\begin{minipage}[t]{.28\textwidth}
\begin{center}
{\large{\bf{Nets Hawk Katz}}} \\
\vskip0.15cm
\footnotesize{Rice University}
\end{center}
\end{minipage}
\hfill
\noindent
\begin{minipage}[t]{.28\textwidth}
\begin{center}
{\large{\bf{Shukun Wu}}} \\
\vskip0.15cm
\footnotesize{Indiana University Bloomington}
\end{center}
\end{minipage}
\hfill
\noindent
\begin{minipage}[t]{.28\textwidth}
\begin{center}
{\large{\bf{Joshua Zahl}}} \\
\vskip0.15cm
\footnotesize{The University of British Columbia} 
\end{center}
\end{minipage}

\vspace{1cm}


\begin{center}
\noindent \em{Communicated by Larry Guth}
\end{center}
\vspace{1cm}


\noindent \textbf{Abstract.} \textit{We prove that every Kakeya set in $\mathbb{R}^3$ formed from lines of the form $(a,b,0) + \operatorname{span}(c,d,1)$ with $ad-bc=1$ must have Hausdorff dimension $3$; Kakeya sets of this type are called $SL_2$ Kakeya sets. This result was also recently proved by F\"assler and Orponen using different techniques. Our method combines induction on scales with a special structural property of $SL_2$ Kakeya sets, which says that locally such sets look like the pre-image of an arrangement of plane curves above a special type of map from $\mathbb{R}^3$ to $\mathbb{R}^2$, called a twisting projection. This reduces the study of $SL_2$ Kakeya sets to a Kakeya-type problem for plane curves; the latter is analyzed using a variant of Wolff's circular maximal function.}
\vskip0.3cm

\vspace{0.5cm}


\section{Introduction}\label{introSection}
Let $\mathcal L_{SL_2}$ be the set of lines in $\RR^3$ that can be written in the form $\ell_{(a,b,c,d)}=(a,b,0) + \operatorname{span}(c,d,1)$ with $ad-bc=1$, and let $\mathcal{L}_{SL_2}^*$ be the set of lines in $\mathcal L_{SL_2}$, plus those that can be written in the form $(0,0,t)+\operatorname{span}(c,d,0)$. Recall that a Kakeya set in three dimensions is a compact set $K\subset\RR^3$ that contains a unit line segment pointing in every direction. A $SL_2$ Kakeya set is a compact set $K\subset\RR^3$ that contains a unit line segment in every direction\footnote{except the direction $(0,0,1)$, since no line in $\mathcal L_{SL_2}^*$ points in this direction}, that comes from a line in $\mathcal L_{SL_2}^*$. The Kakeya set conjecture in three dimensions asserts that every Kakeya set in $\RR^3$ has Minkowski and Hausdorff dimension 3. In \cite[Conjecture 1.3]{WZ}, Wang and the third author stated a special case of this conjecture for $SL_2$ Kakeya sets.

\begin{conj}\label{sl2Conj}
Every $SL_2$ Kakeya set in $\RR^3$ has Minkowski and Hausdorff dimension 3. 
\end{conj} 

Conjecture \ref{sl2Conj} was recently resolved by F\"assler and Orponen \cite{FO}. We give an alternate proof of Conjecture \ref{sl2Conj}. Our main contribution is a single-scale volume bound for the thickened neighborhood of unions of lines. In what follows, we identify each line $\ell_{(a,b,c,d)}\in \mathcal{L}_{SL_2}$ with the corresponding point $(a,b,c,d)\in\RR^4$, and we define the distance $d(\ell,\ell')$ to be the Euclidean distance between the corresponding points.


\begin{thm}\label{mainThm}
For all $\eps>0,\ R>0$, there exists $M=M(\eps)>0$ and $\delta_0=\delta_0(\eps,R)>0$ so that the following holds for all $\delta\in (0,\delta_0]$. Let $L\subset \mathcal{L}_{SL_2}\cap B(0,R)$. Suppose that $L$ satisfies the two-dimensional ball condition 
\begin{equation}\label{ballCondition}
\#(L \cap B) \leq (r/\delta)^2\quad\textrm{for all}\ r\in[\delta,1],\ \textrm{and all balls}\ B\ \textrm{of radius}\ r.
\end{equation}
Let $E\subset B(0,R)$ be a union of $\delta$-cubes, and suppose $|\ell\cap E|\geq \lambda$ for each $\ell\in L$. Then
\begin{equation}\label{mainThmBigVolumeEstimate}
|E|\geq \delta^{\eps} \lambda^{M} (\delta^2\# L).
\end{equation}
\end{thm} 
\noindent Theorem \ref{mainThm} implies Conjecture \ref{sl2Conj} by standard discretization arguments (observe that if $K$ is compact, then for sufficiently large $R$ we have that $K\subset B(0,R)$, and all lines $\ell\in\mathcal{L}_{SL_2}^*$ intersecting $K$ and making angle $\leq 1/2$ with the direction $(0,0,1)$ will be contained in $\mathcal{L}_{SL_2}\cap B(0,R)$).

\medskip

\noindent{\bf Remarks}.\\
1. The Kakeya maximal function conjecture asserts than an estimate of the form \eqref{mainThmBigVolumeEstimate} holds for $M=3$ (for lines pointing in $\delta$-separated directions). If the lines in $L$ are not restricted to $\mathcal{L}_{SL_2}$, then a focusing example shows that $M=3$ is best possible. However, this example cannot be realized using lines from $\mathcal{L}_{SL_2}$; the analogous focusing example with lines from $\mathcal{L}_{SL_2}$ suggests that \eqref{mainThmBigVolumeEstimate} might be true with $M=2$. 

\medskip

\noindent 2. If the lines in $L$ point in $\delta$-separated directions, or satisfy the Wolff axioms, then the ball condition \eqref{ballCondition} is automatically satisfied. The converse is not true---for example, if $L_0$ is the set of lines in $\RR^3$ of the form $(n\delta, 0, 0) + \operatorname{span}(0, m\delta, 0)$ for integers $1\leq n,m\leq\delta^{-1}$, then the lines in $L_0$ satisfy a two-dimensional ball condition, but they do not point in different directions, nor do they satisfy the Wolff axioms. And indeed, no estimate of the form \eqref{mainThmBigVolumeEstimate} is possible for the set $L_0$, since $\bigcup_{L_0}\ell$ is contained in the $\{z=0\}$ plane. This example does not contradict Theorem \ref{mainThm} because the lines in $L_0$ are not contained in $\mathcal{L}_{SL_2}$. The reason we do not require the set of lines $L\subset\mathcal{L}_{SL_2}$ to be direction separated (or to satisfy the Wolff axioms) is because the variety $\mathcal{L}_{SL_2}$ automatically enforces the Wolff axioms---any set $L\subset\mathcal{L}_{SL_2}$ satisfying the ball condition \eqref{ballCondition} must necessarily satisfy the Wolff axioms. It is possible that the analogue of Theorem \ref{mainThm} holds when $\mathcal{L}_{SL_2}$ is replaced by any set $\mathcal{L}\subset\RR^4$ that enforces the Wolff axioms in this fashion.

\subsection{The F\"assler-Orponen proof, and its converse}
As noted above, F\"assler and Orponen recently solved Conjecture \ref{sl2Conj}. An important ingredient in their proof is the following special case of a result by Gan, Guo, Guth, Harris, Maldague, and Wang, \cite{GGGHMW}.
\begin{thm}[GGGHMW]\label{GGGHMWThm}
Let $\gamma(\theta) = \frac{1}{\sqrt 2}(\cos\theta, \sin\theta, 1)$ and for each $\theta\in [0,1]$, let $\pi_\theta$ be the projection to $\gamma(\theta)^\perp$. Let $E\subset\RR^3$ be a Borel set. Then $\dim\pi_\theta(E) = \min(\dim E, 2)$ for a.e.~$\theta\in[0,1]$. 
\end{thm}
The authors in \cite{GGGHMW} proved a more general result that resolved a conjecture of F\"assler and Orponen \cite{FA2} concerning the behavior of orthogonal projections in a restricted set of directions. However, the version of Theorem \ref{GGGHMWThm} stated above was sufficient for F\"assler and Orponen to prove Conjecture \ref{sl2Conj}. Theorem \ref{GGGHMWThm} was proved using decoupling. 

In Section \ref{reversingFOSec}, we show that F\"assler and Orponen's proof can be reversed, and hence Theorem \ref{mainThm} implies Theorem \ref{GGGHMWThm}. This yields a new proof of Theorem \ref{GGGHMWThm} that does not use Fourier analysis---our proof of Theorem \ref{mainThm} uses bounds on a variant \cite{PYZ} of Wolff's circular maximal function, and these bounds in turn are obtained using tools from incidence geometry.


\subsection{The $SL_2$ Kakeya conjecture and the Kakeya set conjecture}
In this section we will discuss the motivation for Conjecture \ref{sl2Conj}, and explain how it relates to the Kakeya conjecture. 
\subsubsection{Kakeya sets and near misses}
One reason the Kakeya conjecture is difficult is that there are objects that have some properties in common with Kakeya sets (such as containing many unit line segments) that have Hausdorff dimension (or some analogue thereof) less than 3. If such an object has dimension $d<3$, then we will call it a \emph{near miss at dimension $d$} for the Kakeya conjecture. 

Two near misses for the Kakeya conjecture are the \emph{Heisenberg group} example $\mathbb{H}$ from \cite{KLT} and the $SL_2$ example $\mathfrak{S}$ from \cite{KZ}; both of these are near misses at dimension $5/2$. Neither of these objects is a literal counter-example to the Kakeya conjecture, because the $\mathbb{H}$ is a subset of $\CC^3$, while $\mathfrak{S}$ is a subset of $( F_p[x]/(x^2) )^3$ (these objects also fail to contain a line pointing in every direction, but they do satisfy a slightly weaker condition known as the Wolff axioms). In \cite{KZ}, the first and third authors showed that any (hypothetical) Kakeya set with Hausdorff dimension $5/2$ must closely resemble either $\mathbb{H}$ of $\mathfrak{S}$; informally, this means that the Heisenberg group example and the $SL_2$ example are the only near misses at dimension $5/2$. They then showed that neither of these examples could be realized in $\RR^3$ at dimension $5/2$. The conclusion was that every Kakeya set in $\RR^3$ must have Hausdorff dimension at least $5/2+c$, for some small constant $c>0$. 

The analysis in \cite{KZ} left open the possibility that an object like the $SL_2$ example could be realized in $\RR^3$ at some dimension $5/2+c \leq d < 3$. Conjecture \ref{sl2Conj} asserts that this cannot happen.

\subsubsection{$SL_2$ Kakeya sets and plaininess}
A defining feature of $\mathcal L_{SL_2}^*$ is that for each point $p=(p_x,p_y,p_z) \in \RR^3$, the plane $\Pi(p)=p+(p_y,-p_x,1)^\perp$ has the following property.
\begin{equation}\label{planeDistribution}
\{\ell \in \mathcal{L}_{SL_2}^*\colon p\in \ell\} = \{\ell\subset \Pi(p)\colon p\in \ell\}.
\end{equation}
There is also a partial converse to this statement: Let $\mathcal{L}$ be a set of lines in $\RR^3.$ Suppose that for each point $p\in\RR^3$ there is a plane $H(p)$ containing $p$, so that \eqref{planeDistribution} is satisfied with $\mathcal{L}$ in place of $\mathcal{L}_{SL_2}^*$, and $H(p)$ in place of $\Pi(p)$. Then either $\mathcal{L}$ arises from a foliation of $\RR^3$ by parallel planes, or $\mathcal{L}$ is the image of $\mathcal{L}_{SL_2}^*$ after applying a projective transformation to $\RR^3$.

A Kakeya set where all the lines passing through each point $p\in K$ are coplanar is called a ``plany Kakeya set,''  and \eqref{planeDistribution} says that every $SL_2$ Kakeya set is automatically a plany Kakeya set. Plany Kakeya sets play an important role in the Kakeya set conjecture because one of the consequences of the Bennett-Carbery-Tao multilinear Kakeya theorem \cite{BCT} is that a (hypothetical) Kakeya set $K\subset\RR^3$ with Hausdorff dimension less than 3 must exhibit a weak form of planiness. While $SL_2$ Kakeya sets are not the only type of plany Kakeya sets, they are a useful model for plany Kakeya sets, and some of the defining features of $SL_2$ Kakeya sets (such as smoothness of the plane map $p\mapsto \Pi(p)$ and the relationship between lines in $SL_2$ Kakeya sets and reguli) have proved useful when exploring the Kakeya set conjecture \cite{KZ, WZ}.

\subsection{A sketch of the proof}
We will briefly describe a few key geometric features of $SL_2$ Kakeya sets, and explain how these properties lead to the proof of Theorem \ref{mainThm}. To simplify notation in what follows, we begin by fixing the number $R\geq 1$ from the statement of Theorem \ref{mainThm}, and we will work in the metric space $B(0,R)\subset\RR^3$ (and occasionally $B(0,R)\subset\RR^2$) with the usual Euclidean metric. In particular, a ``line'' $\ell$ refers to the restriction of $\ell$ to $B(0,R)$. We also abuse notation and replace $\mathcal{L}_{SL_2}$ by $\mathcal{L}_{SL_2}\cap B(0,R)$

Let $L$ be a $\delta$-separated subset of $\mathcal L_{SL_2}$ and let $E$ be a union of grid-aligned $\delta$-cubes that contains $\bigcup_{\ell\in L}\ell$. For this proof sketch, we will suppose that each $\delta$-cube $Q\subset E$ is intersected by at least two lines from $L$ that point in $\sim 1$ separated directions (for example, we exclude the situation where all of the lines intersecting a $\delta$-cube made small angle with a common vector $v$). 

\subsubsection{Regulus strips}
For each $p\in \RR^3\backslash\{x=y=0\}$, define $\hat\ell(p) = p+\operatorname{span}(p_x, p_y, 0)$. Observe that $\hat\ell(p)$ is contained in $\Pi(p)$ (though $\hat\ell(p)$ is not contained in $\mathcal{L}_{SL_2}$). For each line $\ell \in \mathcal L_{SL_2}$, define
\begin{equation}\label{RellDefn}
R(\ell) = \bigcup_{p\in \ell}\hat\ell(p).
\end{equation}
$R(\ell)$ is a doubly-ruled surface (specifically, a hyperbolic paraboloid), and the union \eqref{RellDefn} illustrates one of its rulings. Define 
\begin{equation}
S(\ell) = N_{\sqrt\delta}(\ell)\ \cap N_\delta (R(\ell) ).
\end{equation}
We will call $S(\ell)$ a \emph{regulus strip}. Regulus strips played an important role in the analysis of Kakeya sets with Hausdorff dimension close to $5/2$ by the first and third authors in \cite{KZ}. It was in this context that questions related to $SL_2$ Kakeya sets were first investigated. 

The set $S(\ell)$ has another interpretation. If we identify $\RR^3$ with the first Heisenberg group, then $S(\ell)$ is comparable to the $\sqrt\delta$ neighborhood of $\ell$ in the Kor\'anyi metric. An key observation about $SL_2$ Kakeya sets is that for each line $\ell\in L$, a substantial fraction of $S(\ell)$ is contained in $E$. Phrased differently: for each $\ell\in L$, the (Euclidean) $\delta$-neighborhood of $\ell$ is contained in $E$ by construction. But more is in fact true: a substantial fraction the (Kor\'anyi) $\sqrt\delta$-neighborhood of $\ell$ is also contained in $E$. This observation follows from Cordoba's $L^2$ Kakeya argument \cite{cordoba}, and a precise version is given in Section \ref{regStripLargeIntersectionEl}.

\subsubsection{Twisted projections and circular maximal functions}
For $\ell_{(a,b,c,d)}\in \mathcal{L}_{SL_2}$, we define the twisted projection $\pi_\ell\colon\RR^3\to\RR^2$ with \emph{core-line} $\ell$ by
\begin{equation}\label{twistedProjectionEll}
\pi_{\ell}(x,y,z)=\big(z,\  (a,b,c,d)\cdot(y, -x, yz, -xz)\big).
\end{equation}
$\pi_{\ell}(\ell)$ is the $x$-axis in $\RR^2$, and for $\ell'\in\mathcal{L}_{SL_2}$, $\pi_{\ell}(\ell')$ is the graph of a degree-two polynomial $f$. For $\ell$ fixed, this allows us to identify each line $\ell'\in \mathcal{L}_{SL_2}$ with a degree-two polynomial, and this map is bilipschitz. Moreover, if $d(\ell,\ell')\leq\sqrt\delta$, then $\pi_{\ell}(S(\ell))$ is comparable to the $\delta$-neighborhood of $\pi_\ell(\ell')$. 

This suggests that we should examine the set of lines in $L\cap N(\ell,\sqrt\delta)$; each such line $\ell'$ has a regulus strip $S(\ell')$, with $|S(\ell)\cap E|$ almost as large as $|S(\ell)|$, and each of these regulus strips projects to the $\delta$-neighborhood of a plane curve. By analyzing these plane curves, we can estimate the volume of $E\cap N_{\sqrt\delta}(\ell)$. To do this, we will use a variant of the Wolff circular maximal function \cite{Wo}. Informally, our result is as follows. 

\begin{prop}[A dichotomy for unions of plane curves, informal version]\label{informalProp}$\phantom{1}$\\
Let $\mathcal{C}$ be a set of degree-two plane curves. Then after replacing $\mathcal{C}$ by a large subset, at least one of the following must hold. 
\begin{itemize}
    \item[(A)] The $\delta$-neighborhoods of the curves in $\mathcal{C}$ are almost disjoint.
    \item[(B)] There is a scale $\rho >\!\!> \delta,$ so that for each curve $c_0\in \mathcal{C}$, the $\rho$-neighborhood of $c_0$ has large intersection with $\bigcup_{c\in\mathcal{C}}N_\delta(c)$.
\end{itemize}
\end{prop}
\noindent A more precise version of Proposition \ref{informalProp} appears in Section \ref{PYZSection}.

Using Proposition \ref{informalProp} and the geometric observations discussed above, we can prove Theorem \ref{mainThm} by induction on scale. We first cover $L$ by balls of radius $\delta^{1/2}$; each ball corresponds to a $1\times\sqrt\delta\times\sqrt\delta$ tube in $\RR^3$. We apply the twisted projection \eqref{twistedProjectionEll} to the lines inside each tube, using the center of the tube as the core-line; this gives us a set $\mathcal{C}$ of plane curves. Each plane curve $c\in \mathcal{C}$ corresponds to a line, whose associated regulus strip has large intersection with $E$. We apply Proposition \ref{informalProp} to the curves from $\mathcal{C}$. 

Suppose that Conclusion (A) from Proposition \ref{informalProp} holds. Then the regulus strips described above are almost disjoint, and since each regulus strip has large intersection with $E$, this gives us a favorable lower bound on the volume of $E$ inside a typical tube of diameter $\sqrt\delta$. We have moved from scale $\delta$ to scale $\sqrt\delta$; we can now apply induction on scales. 

Suppose instead that Conclusion (B) from Proposition \ref{informalProp} holds. Then the regulus strips described above cluster tightly. But this means that for a typical line $\ell\in L$, the $\rho$-neighbor\-hood of $\ell$ has large (i.e.~almost full) intersection with $E$. We have moved from scale $\delta$ to scale $\rho>\!\!>\delta$; we can now apply induction on scales. 

We encounter several technical difficulties when executing the above strategy. Chief among them is the issue that when we move from scale $\delta$ to a coarser scale, we obtain a new collection of lines $\tilde L,$ and a coarser version of our Kakeya set, which we denote by $\tilde E.$ We originally had $|\ell\cap E|\geq\lambda$ for all lines $\ell$, but at a coarser scale, the quantity $|\tilde\ell\cap\tilde E|$ might have decreased. We need to carefully control how much is lost at each step of the induction.

\subsection{Thanks}
The authors would like to thank Katrin F\"assler and Tuomas Orponen for comments and suggestions on an earlier version of this manuscript, and the anonymous referee for his or her careful reading and corrections. Nets Hawk Katz was supported by a Simons Investigator Award. Joshua Zahl was supported by a NSERC Discovery Grant. 

\section{Geometric Lemmas}
In this section we will explore some of the geometric properties of $\mathcal{L}_{SL_2}$ that were described in Section \ref{introSection}. We maintain the convention of working in the metric space $B(0,R)\subset\RR^3$ (or occasionally $B(0,R)\subset\RR^2$), and we restrict attention to the subset of $\mathcal{L}_{SL_2}$ contained in $B(0,R)$. 

We write $X\lesssim Y$ (equivalently, $Y\gtrsim X$ or $X = O(Y)$ or $Y=\Omega(X)$) to mean that there is a constant $C$ (which is allowed to depend on $R$) so that $X\leq CY$. In the arguments that follow, we will frequently use dyadic pigeonholing on some set $A$ to extract a subset $A'\subset A$ with $\#A'\gtrsim|\log\delta|^{-1}\#A$, so that all elements of $A'$ share a common statistic (for example, they might all have comparable radii, up to a factor of two). To simplify notation, we will write $X\lessapprox Y$ (or equivalently $Y\gtrapprox X$ or $X=O^*(Y)$ or $Y = \Omega^*(X)$) to mean there is an absolute constant $C$ (in practice $C\leq 10$ will suffice) so that $X\lesssim |\log\delta|^C Y$. If there is risk of ambiguity, we will write $\lessapprox_\delta$ to emphasize the role of the variable $\delta$. Finally, we use $N_t(X)$ to denote the $t$-neighborhood of the set $X$.

\subsection{Twisted projections and regulus strips}

We begin by exploring some of the properties of the twisted projection map \eqref{twistedProjectionEll}. Let $\ell=\ell_{(a,b,c,d)}\in\mathcal{L}_{SL_2}$. Observe that for each point $p=p(t) = (a+ct, b+dt, t)\in \ell$, the image of the line $\hat\ell(p) = (a+ct, b+dt, t)+\operatorname{span}(a+ct, b+dt, 0)$ under $\pi_{\ell}$ is the point $\pi_{\ell}(p)$ 
\[
\pi_{\ell}\big(\hat\ell(p) \big)=\pi_{\ell}(p)=(p_z,0)\quad\textrm{for all}\ \ell\in \mathcal{L}_{SL_2},\ p\in \ell.
\]

This allows us to partition the neighborhood of $\ell$ into (slightly distorted) thin rectangular prisms, so that the image of each prism under $\pi_{\ell}$ is a thin rectangle. The precise statement is as follows. 

\begin{defn}
Let $\ell\in\mathcal{L}_{SL_2}$ and let $0<\delta\leq t \leq\delta^{1/2}\leq 1$. A \emph{prism decomposition of $N_t(\ell)$ into prisms of thickness $\delta$} is a cover of $N_t(\ell)$ by $O(1)$-overlapping parallelepipeds $P\subset N_{2t}\ell$, which have dimensions $\delta/t \times t \times \delta$ in the directions $v_1,v_2,v_3$. The vectors $v_1,v_2,v_3$ determined as follows. Let $p=(p_x,p_y,p_z)$ be the center of $P$. 
\begin{itemize}
\item $v_1\cdot(p_y, -p_x, 1)=0$, and $v_1$ makes angle $\leq t$ with the direction of $\ell$. 
\item $v_2$ is parallel to $(p_x, p_y, 0)$ (this is well-defined, since $\ell\in \mathcal{L}_{SL_2}\cap B(0,R)$ forces the latter vector to be non-zero, provided $t$ is sufficiently small). 
\item $v_3$ is parallel to $(p_y, -p_x, 1)$.
\end{itemize}
\end{defn}

It is straightforward to verify that for $0<\delta\leq t \leq\delta^{1/2}\leq 1$, a prism decomposition of $N_{t}(\ell)$ always exists.

\begin{lem}\label{projectionOfPrism}
Let $\ell\in\mathcal{L}_{SL_2}$, let $0<\delta\leq t \leq\delta^{1/2}\leq 1$, and let $\mathcal{P}$ be a prism decomposition of $N_{t}(\ell)$. Then $\{\pi_{\ell}(P)\colon P\in\mathcal{P}\}$ is a $O(1)$-overlapping collection of sets, each of which is comparable to a rectangle of dimensions $\delta/ t \times\delta$. The union of these sets cover the $t$ neighborhood of the $x$-axis, and are contained in the $2t$ neighborhood of the $x$-axis. 
\end{lem}
The proof of Lemma \ref{projectionOfPrism} is a straightforward computation, and is omitted. Observe that the (truncated) regulus strip
\begin{equation}\label{defnTruncatedReg}
S_\delta^t(\ell) = N_{t}(\ell) \cap N_\delta(R(\ell))
\end{equation}
is compatible with the prism decomposition of $N_{t}(\ell)$, in the sense that $S_\delta^t(\ell)$ intersects $\sim t/\delta$ prisms from $\mathcal{P}$, and each of these prisms is contained in the $O(\delta)$-neighborhood of $S_\delta^t(\ell)$. 

Since the longest axis of each prism $P\in\mathcal{P}$ makes angle $\leq t$ with the direction of $\ell$, this means that if $\ell,\ell'\in\mathcal{L}_{SL_2}$ with $d(\ell,\ell')\leq t$, then the prism decompositions of $N_t(\ell)$ and $N_{t}(\ell')$ are comparable where the overlap. In particular, $S_\delta^t(\ell')$ intersects $ O(t/\delta)$ prisms from the prism decomposition of $N_{t}(\ell)$. This has the following consequence. 

\begin{lem}\label{twistedProjectionRespectsRegulusStrips}
Let $0<\delta<t\leq\delta^{1/2}\leq 1$ and let $\ell,\ell' \in\mathcal{L}_{SL_2}$ with $d(\ell,\ell')\leq t$. Then there are constants $c$ and $C$ so that
\[
 N_{c\delta}\big(\pi_\ell(\ell')\big) \subset \pi_\ell \big(S_\delta^t(\ell') \big)\subset  N_{C\delta}\big(\pi_\ell(\ell')\big).
\]
\end{lem}

To make use of Lemma \ref{twistedProjectionRespectsRegulusStrips}, we should compute $\pi_\ell(\ell')$. We record the result of this computation below. 

\begin{lem}\label{twistedProjectionOfSL2Line}
Let $\ell=\ell_{(a,b,c,d)},\ \ell'=\ell'_{(a',b',c'd')}\in \mathcal{L}_{SL_2}$. Then $\pi_\ell(\ell')$ is the graph of the degree-two polynomial $f(x) = Ax^2 + Bx + C$, where
\begin{equation}\label{defnABC}
A = (c,d)\cdot(d',-c'),\quad B = (a,b,c,d)\cdot(d',-c',b',-a'),\quad C = (a,b)\cdot(b',-a').
\end{equation}
\end{lem}

By differentiating the expression \eqref{defnABC} and using compactness (recall that we restrict consideration to $\mathcal{L}_{SL_2}\cap B(0,R)$), we conclude the following.
\begin{lem}\label{preserveSpacing}
There exists $\rho_0>0$ and $K_0$ so that for all $\ell\in\mathcal{L}_{SL_2}$, the map $F_\ell \colon \mathcal{L}_{SL_2}\to\RR^3$ given by $F_\ell(\ell') =(A,B,C)$ (with $A,B,C$ as defined in \eqref{defnABC}) is bilipschitz on $N(\ell, \rho_0)$, with bilipschitz constant at most $K_0$. 
\end{lem}

\subsection{Lines and shadings}
Let $L\subset\mathcal{L}_{SL_2}$ be a set of lines. For each $\ell\in L$, let $Y(\ell)$ be a union of axis-aligned $\delta$-cubes that intersect $\ell$. The set $Y(\ell)$ is called a shading of $\ell$, and if $|Y(\ell)|\geq\lambda\delta^2$, then we say $Y(\ell)$ is $\lambda$-dense. 

We call the pair $(L,Y)$ an \emph{arrangement of lines and their associated shading}. We will sometimes write $(L,Y)_\delta$ to emphasize the scale $\delta$. For such an arrangement, define $E_L=\bigcup_{\ell\in L}Y(\ell)$; this is a union of axis-aligned $\delta$-cubes. For each $x\in\RR^3$, define $L(x) = \{\ell\in L\colon x\in Y(\ell)\}$. Note that $E_L$ and $L(x)$ depend on the choice of shading $Y$, but this shading will always be apparent from context. Note as well that $L(x)$ is the same for all $x$ in the same axis-aligned $\delta$-cube.  

\begin{defn}\label{regularShadingDef}
Let $\ell\in\mathcal{L}_{SL_2}$. We say a shading $Y(\ell)$ is \emph{regular} if for each $x\in Y(\ell)$ and each $\delta\leq r\leq R$, we have 
\[
|Y(\ell)\cap B(x,r)| \geq \frac{r|Y(\ell)|}{C|\log\delta|} .
\]
Here $C$ is a constant that depends on the quantity $R$ from the statement of Theorem \ref{mainThm}. 
\end{defn}
\begin{lem}\label{lemRegularShading}
Let $\ell\subset B(0,R)$ be a line segment and let $Y(\ell)$ be a shading. If the constant $C=C(R)$ from Definition \ref{regularShadingDef} is chosen appropriately, then there is a regular refinement $Y'(\ell)\subset Y(\ell)$ with  $|Y'(\ell)|\geq \frac{1}{2}|Y(\ell)|$. 
\end{lem}
\begin{proof}
Let $M=\lceil \log_2(R/\delta)\rceil$. Define $Y_{-1} = Y(\ell)$ and let $\alpha=|Y(\ell)|\delta^{-2}$. For each $i=0,\ldots,M-1$, divide $\ell$ into $2^{M-i}$ equal line-segments of length $\sim R 2^{i-M}$; denote this set of line segments by $\{J\}$. Let $Y_{i} = \bigcup_J Y_{i-1}^J$, where  $Y_{i-1}^J$ is the union of $\delta$-cubes $Q\subset Y_{i-1}$ that intersect $J$, and the union is taken over those segments $J$ satisfying $|Y_{i-1}^J|\geq \frac{\alpha|J|\delta^2}{4\log_2(R/\delta)}$. Note that for each set $Y_{i-1}^J$, we have that either $Y_{i-1}^J\subset Y_i$ or $Y_{i-1}^J\cap Y_i=\emptyset.$ In particular, if $x\in Y_{M-1}$ and if $x\in Y_i^J$ for some index $i$ and some line-segment $J$, then $Y_i^J\subset Y_{M-1}$. Define $Y'(\ell) = Y_{M-1}$. 
By our definition of $Y_i$, we have $|Y_i|\geq |Y_{i-1}| - \frac{|Y(\ell)|}{4\log_2(R/\delta)}$, and hence $|Y'(\ell)| = |Y_{M-1}| \geq \frac{1}{2}|Y(\ell)|$.

We will verify that $Y'$ is a regular shading. Let $x\in Y'$ and let $r\in[\delta,R]$. Let $i$ be the biggest index with $R2^{i-M}\leq r/2$. Then $x\in Y_i^J$ for some line-segment $J$ of length $\leq R2^{i-M}\leq r/2$, and hence $J\subset B(x,r)$. But this means that 
\[
|Y'(\ell)\cap B(x,r)|\geq |Y_i^J| \gtrsim  \frac{\alpha|J|\delta^2}{\log_2(R/\delta)} \gtrsim  \frac{r|Y(\ell)|}{|\log\delta|}.\qedhere
\]
\end{proof}

\begin{defn}
If $L$ is a set of lines and $Y(\ell)$ is a regular shading of $\ell$ with $|Y(\ell)|\geq\lambda \delta$ for each $\ell\in L$, then we say that $(L,Y)_\delta$ is a set of lines with a regular, $\lambda$-dense shading. 
\end{defn}

\subsection{Regulus strips have large intersection with $E_L$}\label{regStripLargeIntersectionEl}

Recall that the regulus strip $S_\delta^t(\ell)$ is covered by about $t/\delta$ prisms from the prism decomposition of $N_t(\ell)$. $\ell$ intersects a typical prism from this collection in a line segment of length about $|\ell\cap P|\sim \delta/t$, and if $Y(\ell)$ is a shading with $|Y(\ell)|\sim\delta^2$, then we might expect a typical intersection to have size roughly $\delta^3 /t.$ Note that if $t>\!\!>\delta$, then this is much smaller than the entire prism, which has volume $|P|\sim \delta^2$. The next result says that under certain circumstances, $|E_L\cap P|$ is much larger than $|Y(\ell)\cap P|$; in fact, $|E_L\cap P|$ is almost as large as $|P|$. To make this statement precise, we need the following definition.

\begin{defn}\label{defnRx}
Let $(L,Y)_\delta$ be a set of lines, and their associated shading. For each $x\in \RR^3$, define 
\[
r_L(x) = \delta+\sup\{\angle(\ell,\ell')\colon \ell,\ell' \in L(x)\}.
\]
If the arrangement $(L,Y)_\delta$ is apparent from context, we will write $r(x)$ in place of $r_L(x)$. 
\end{defn}

\begin{lem}\label{mostlyFullRegulusStrip}
Let $0<\delta<1$, let $r\in [4\delta,1]$, and let $t\in [\delta, (\delta r)^{1/2}]$. Let $(L,Y)_\delta$ be a set of lines with a regular, $\lambda$-dense shading. Let $\ell \in L$, let $\mathcal{P}$ be the prism decomposition of $N_t(\ell)$ of thickness $\delta$, and let $P\in \mathcal{P}$. Suppose that
\begin{equation}\label{goodPointsOnEll}
\big| P \cap Y(\ell) \cap \{ x\colon r \leq r(x) \leq 2r\}\big|\geq \alpha\delta^3t^{-1},
\end{equation}
for some $\alpha>0$. (see Remark \ref{interpretGoodPointsOnEllRem} below). Then there is a number $C=O(1)$ so that
\begin{equation}
\big| CP \cap E_L\big| \gtrapprox  \alpha \lambda^2 |P|,
\end{equation}
where $CP$ denotes the $C$-fold dilate of $P$. 
\end{lem}
\begin{rem}\label{interpretGoodPointsOnEllRem}
To interpret \eqref{goodPointsOnEll}, note that the set on the LHS is contained in a tube segment of dimensions $\delta/ t \times \delta \times\delta$, and thus this segment has volume roughly $\delta^3t^{-1}$.
\end{rem}
\begin{proof}
Let $X$ be the set on the LHS of \eqref{goodPointsOnEll}. Let $x_1,\ldots,x_M$, $M \gtrsim \alpha r/t$ be points in $X$, arranged so that $|x_i-x_j| \geq 100 |i-j| \delta/r$. For each index $i$, select $\ell_i\in L(x_i)$ with $\angle(\ell,\ell_i) \in [r/4, 2r]$; such a tube $\ell_i$ must exist, since $x_i\in X$ implies that $\angle(\ell,\ell_i)\leq 2r$ for each line $\ell_i\in L(x_i)$. On the other hand, there are at least two lines $\ell_i,\ell_i'\in L(x_i)$ with $\angle(\ell_i,\ell_i')\geq r-\delta\geq r/2$, so by the triangle inequality at least one of them must make angle $\geq r/4$ with $\ell$. 

Note that the sets $\{Y(\ell_i)\}$ are contained in the $2\delta$ neighborhood of their associated lines $\ell_i$, each of which make an angle $\geq r/4$ with $\ell$. We conclude that the spacing condition on $|x_i-x_j|$ ensures that the sets $\{Y(\ell_i)\cap \ell\}$ are disjoint. 

Define $T = N_{t}(\ell)$. Since $\angle(\ell,\ell_i) \in [r/4, 2r]$ and $t\leq(\delta r)^{1/2}$, we have that $\ell_i\cap T \cap B(x_i, \delta/t)$ is a line segment of length $\sim t/r$, i.e.~the line segment $\ell_i$ exits the cylinder $T$ before (or at least not long after) it exits the ball $B(x_i, \delta/t)$. Since the prism $P$ has dimensions $\delta/t\times t\times \delta$, this means that if the constant $C$ is chosen appropriately, we have that $\ell_i \cap CP$ is a line segment of length $\sim t/r$. Fix this choice of $C$, and define $\tilde P = CP$. Since $Y(\ell_i)$ is a regular shading, we therefore have 
\begin{equation}\label{Y1TiCapNr}
|Y(\ell_i)\cap \tilde P |\gtrapprox \lambda  (t/r)\delta^2.
\end{equation}
 
If $i\neq j$ and $Y(\ell_i)\cap Y(\ell_j) \cap  \tilde P \neq\emptyset$, then since $Y(\ell_i)$ and $Y(\ell_j)$ intersect $\ell$ at points $x_i$ and $x_j$ respectively, with $|x_i-x_j|\geq 100|i-j| \delta/r$, we can use the law of sines to estimate $\angle(\ell_i,\ell_j)\gtrsim \frac{\delta r}{t}|i-j|$, and hence  
\[
|\tilde P \cap Y(\ell_i)\cap Y(\ell_j)|\leq |T \cap N_{2\delta}(\ell_i)\cap N_{2\delta}(\ell_j)|\leq \frac{\delta^2 t}{r(|i-j|+1)}.
\]
Thus
\[
\Big \Vert \sum_{i=1}^M \chi_{Y(\ell_i)}\Big\Vert_{L^2(\tilde P)}^2 \lesssim  \delta^2 t r^{-1}\sum_{1\leq i,j\leq M} \frac{1}{|i-j|+1} \lesssim \delta^2 t r^{-1} M \log M. 
\]
On the other hand, by \eqref{Y1TiCapNr} we have
\[
\Big \Vert \sum_{i=1}^M \chi_{Y(\ell_i) }\Big\Vert_{L^1(\tilde P)} \gtrapprox M \lambda t r^{-1}\delta^2.
\]
By Cauchy-Schwartz, we have 
\[
|CP\cap E_L|\geq \Big| \tilde P \cap \bigcup_{i=1}^M Y(\ell_i)\Big| \geq  \frac{ \Big \Vert\sum_{i=1}^M \chi_{Y(\ell_i) }\Big\Vert_{L^1(\tilde P)}^2}{\Big \Vert \sum_{i=1}^M \chi_{  Y(\ell_i) }\Big\Vert_{L^2(\tilde P)}^2} \gtrapprox \lambda^2 M tr^{-1}\delta^2  \gtrsim\alpha \lambda^2 |P|.\qedhere  
\]
\end{proof}

For the next lemma, we consider a set of lines contained in the thin neighborhood of a core-line $\ell_0$. We find a dense shading of $\ell_0$ (or more precisely, a line close to $\ell_0$) at a coarser scale $t$. Each $t$-cube in this shading has large intersection with each of the regulus strips $S_\delta^t(\ell)$ coming from the lines close to $\ell_0$. The precise statement is as follows.

\begin{lem}\label{fiberAboveProjectionIsLarge}
Let $0<\delta<1$, let $r\in [4\delta,1]$, and let $t\in [\delta, (\delta r)^{1/2}]$.
Let $(L,Y)$ and $(L',Y')$ be sets of lines with regular, $\lambda$-dense shadings. Suppose that $L'\subset L\cap B$, for some ball $B = B(\ell_0, t)\subset \mathcal L_{SL_2}$. Suppose also that $Y'(\ell')\subset \{x \colon r \leq r_L(x) \leq 2r\}$ for each $\ell'\in L'$. 

Then there exists a line $\tilde\ell\in B(\ell_0,2t)$ and a $\Omega^*(\lambda^4)$-dense shading $\tilde Y$ of $\tilde\ell$ by $t$-cubes (i.e.~$\tilde Y(\tilde \ell)$ is a union of $t$-cubes). After replacing $(L',Y')$ by a $\Omega^*(1)$-density refinement, we have 
\begin{equation}\label{volumeBdEachTcube}
|E_L\cap Q| \gtrapprox \lambda^3 t^2 \Big| \bigcup_{\ell' \in L'} \pi_{\ell_0}(Y'(\ell'))\Big|\quad\textrm{for each $t$-cube}\ Q\subset \tilde Y(\tilde\ell). 
\end{equation}
\end{lem}

\begin{rem}
We should interpret \eqref{volumeBdEachTcube} as follows. The union on the RHS of \eqref{volumeBdEachTcube} is contained in the $O(t)$ neighborhood of the $x$-axis in $B(0,R)\subset\RR^2$. If this union has maximal size and if $\lambda\sim 1$, then the RHS of \eqref{volumeBdEachTcube}  has size roughly $t^3$ (ignoring logarithmic factors), and thus $|E_L\cap Q|\approx |Q|$. 
\end{rem}
\begin{proof}
Let $T = N_t(\ell_0)$ and let $\mathcal{P}_0$ be the prism decomposition of $T$ into prisms of thickness $\delta$. Let $\mathcal{P}_1\subset\mathcal{P}_0$ so that $\sum_{\ell'\in L'}|Y'(\ell')\cap P|$ is about the same (up to a factor of 2) for each $P\in\mathcal{P}_1$, and not too much mass has been lost, i.e.~
\begin{equation}\label{P1PreservesMass}
\sum_{P\in\mathcal{P}_1}\sum_{\ell'\in L'}|Y'(\ell')\cap P| \gtrapprox \sum_{\ell'\in L'}|Y'(\ell')|.
\end{equation}

 We will divide $T$ into  $O(t/\delta)$ tube segments of length $\delta/t$; let $\mathcal{U}$ denote this set of tube segments. After dyadic pigeonholing, select an integer $M$ and a set $\mathcal{U}_2\subset\mathcal{U}$ so that each tube segment $U\in\mathcal{U}_2$ intersects between $M$ and $2M$ prisms from $\mathcal{P}_1$; denote this set of prisms by $\mathcal{P}_2$---we have $\#\mathcal{P}_2\gtrapprox \#\mathcal{P}_1$, and \eqref{P1PreservesMass} remains true with $\mathcal{P}_2$ in place of $\mathcal{P}_1$.  Since $|U\cap Y'(\ell')|\lesssim \delta^3/t$ for each $\ell'\in L'$, we must have have 
 \begin{equation}\label{boundU2}
 \#\mathcal{U}_2\gtrapprox \lambda t/\delta.
 \end{equation}

Each prism $P\in\mathcal{P}_2$ intersects at least one set of the form $Y'(\ell')$, $\ell'\in L'$. But since $Y'(\ell')$ is $\lambda$-dense and regular, and $\ell'\in B(t,\ell_0)$, we have that $|2P \cap Y'(\ell')|\gtrapprox \lambda \delta^3/t$ (recall that $2P\cap\ell'$ is a line-segment of length roughly $\delta/t$), and thus by Lemma \ref{mostlyFullRegulusStrip} (with $\delta,r,t,\lambda$ as above and $\alpha\approx\lambda$) there is a constant $C=O(1)$ so that
\begin{equation}\label{prismVolumeBd}
|E_L\cap CP|\gtrapprox  \lambda^3 |P|\sim\lambda^3 \delta^2. 
\end{equation}

Let $\mathcal{Q}$ be a covering of $\bigcup_{U\in \mathcal{U}_2}U$ by grid-aligned $t$-cubes. Using \eqref{prismVolumeBd}, we have
\begin{equation}\label{sumOverQ}
\sum_{Q\in\mathcal{Q}}\Big|Q\cap E_L \cap \bigcup_{P\in\mathcal{P}_2}CP\Big| \gtrapprox  \lambda^3 \delta^2(\#\mathcal{P}_2)
\gtrsim  \lambda^3 \delta^2 M(\#\mathcal{U}_2).
\end{equation}
Pigeonhole $\mathcal{Q}$ to obtain a set $\mathcal{Q}_3$ with the property that each $Q\in\mathcal{Q}_3$ intersects about the same number (up to a factor of 2) of prisms $\{CP\colon P\in \mathcal{P}_2\}$; each $Q\in\mathcal{Q}_3$ has the same size intersection (up to a factor of 2) with $E_L$; and \eqref{sumOverQ} remains true if the sum on the LHS is taken over $\mathcal{Q}_3$. Let $\mathcal{P}_3\subset\mathcal{P}_2$ denote the set of prisms $P\in \mathcal{P}_2$ so that $CP$ intersects some cube from $\mathcal{Q}_3$. Then \eqref{P1PreservesMass} becomes
\begin{equation}\label{P1PreservesMassP3}
\sum_{P\in\mathcal{P}_3}\sum_{\ell'\in L'}|Y'(\ell')\cap CP| \gtrapprox \sum_{\ell'\in L'}|Y'(\ell')|,
\end{equation}
and \eqref{sumOverQ} becomes
\begin{equation}\label{sumOverQ3}
\sum_{Q\in\mathcal{Q}_3}\Big|Q\cap E_L \cap \bigcup_{P\in\mathcal{P}_3}CP\Big| \gtrapprox  \lambda^3 \delta^2 M(\#\mathcal{U}_2).
\end{equation}
Since $t\leq\delta/t$, each $Q\in\mathcal{Q}_3$ intersects at most two tube segments from $\mathcal{U}$, and hence intersects at most $4M$ prisms from $\mathcal{P}_3$. Since each of these cube-prism intersections has size $O(\delta t^2)$, we conclude that 
\begin{equation}\label{bdPrismCubeIntersection}
\Big|Q\cap E_L \cap \bigcup_{P\in\mathcal{P}_3}CP\Big|\leq \Big|Q\cap \bigcup_{P\in\mathcal{P}_3}CP\Big|\leq 4M\delta t^2\quad\textrm{for each}\ Q\in\mathcal{Q}_3.
\end{equation}
Comparing \eqref{sumOverQ3} and \eqref{bdPrismCubeIntersection}, we have
\[
\#\mathcal{Q}_3\gtrapprox  \frac{ \lambda^3 \delta^2 M(\#\mathcal{U}_2)}{4M\delta t^2}= \lambda^3 \delta t^{-2} (\#\mathcal{U}_2)\gtrapprox  \lambda^4  t^{-1},
\]
where the final inequality used \eqref{boundU2}.

Since the cubes in $\mathcal{Q}_3$ intersect $N_t(\ell_0)$, we can select a line $\tilde\ell\in B(\ell_0,2t)$ that intersect a $\Omega(1)$ fraction of these cubes; fix this line $\tilde\ell$, and let $\tilde Y(\tilde\ell)$ be the union of cubes from $\mathcal{Q}_3$ that intersect $\tilde\ell$.

\medskip

It remains to prove \eqref{volumeBdEachTcube} for a suitable refinement of $(L',Y')$. For each $\ell'\in L'$, let $Y''(\ell')=Y'(\ell')\cap\bigcup_{P\in \mathcal{P}_3}CP$. Let $L''\subset L'$ be those $\ell'\in L'$ with $|Y''(\ell')|
\gtrapprox\lambda\delta^2$. If the implicit constant and power of $|\log\delta|$ is chosen appropriately, then by \eqref{P1PreservesMassP3} we have that $(L'', Y'')$ is a $\Omega^*(1)$-refinement of $(L',Y')$, in the sense that $L''\subset L'$; $Y''(\ell'')\subset Y'(\ell'')$ for each $\ell''\in L''$; the shading $Y''(\ell'')$ is $\Omega^*(\lambda)$ dense, and $\#L''\gtrapprox\#L'$.

We will now verify \eqref{volumeBdEachTcube}. If $P\in\mathcal{P}_3$ is a prism, then
\[
\Big|\pi_{\ell_0}\Big(\bigcup_{\ell''\in L''}Y''(\ell'')\cap CP\Big)\Big|\lesssim \delta^2/t \lessapprox \lambda^{-3}t^{-1}|E_L\cap CP|,
\]
where we used  Lemma \ref{projectionOfPrism} for the first inequality and \eqref{prismVolumeBd} for the second inequality. Thus
\begin{equation}\label{ElCapTildeY}
\begin{split}
| E_L \cap \tilde Y(\tilde \ell) | 
&\gtrsim \Big| E_L\cap \bigcup_{P\in\mathcal{P}_3}P\Big|\\
&\gtrsim  \sum_{P\in\mathcal{P}_3} | E_L \cap CP|\\
&\gtrapprox \lambda^3t \sum_{P\in\mathcal{P}_3}\Big|\pi_{\ell_0}\Big(\bigcup_{\ell'\in L'}Y'(\ell')\cap CP\Big)\Big|\\ 
&\gtrsim \lambda^3t \Big| \bigcup_{\ell''\in L''}\pi_{\ell_0}(Y''(\ell'')) \Big|.
\end{split}
\end{equation}
But each of the $t$-cubes $Q\subset \tilde Y(\tilde\ell)$ (of which there are at most $O(t^{-1})$) has the same size intersection with $E_L$ (within a factor of 2), and hence \eqref{ElCapTildeY} implies \eqref{volumeBdEachTcube}.
\end{proof}


\section{Circular maximal function bounds, and their consequences}\label{PYZSection}
The RHS of \eqref{volumeBdEachTcube} from Lemma \ref{fiberAboveProjectionIsLarge} is the area of a union of thickened plane curves (or more accurately, large subsets of such thickened curves). The following special case of \cite[Theorem 1.7]{PYZ} allows us to bound (from below) the area of this type of set. 
\begin{thm}\label{PYZThm}
For all $\eps>0$, there exists $c_{\eps}>0$ so that the following holds. Let $\delta>0$, let $K\geq 1$, and let $F\subset B(0,R)\subset\RR^3$ satisfy the one-dimensional ball condition
\begin{equation}\label{KBallCondition}
\#(F\cap B)\leq K(r/\delta)\quad\textrm{for all balls}\ B\subset\RR^3\ \textrm{of radius}\ r\geq\delta. 
\end{equation}
For each  $f=(a,b,c)\in F$, let $f^\delta$ be the $\delta$-neighborhood of the graph of the function $ax^2 + bx+c$ above $[-R,R]$ and let $Y(f)\subset f^\delta$, with $|Y(f)|\geq\lambda\delta$. Then
\begin{equation}\label{volumeBd}
\Big|\bigcup_{f\in F}Y(f)\Big| \geq c_{\eps}\delta^\eps  \lambda^3 K^{-1} (\delta\#F).
\end{equation}
\end{thm}

\medskip

\noindent\emph{Remarks}\\
{\bf 1.} Theorem 1.7 from \cite{PYZ} gives an upper bound for the $L^{3/2}$ norm of $\sum_{f\in F}\chi_{f^\delta}$, but this estimate is equivalent to \eqref{volumeBd} by standard arguments.\\
{\bf 2.} Theorem 1.7 from \cite{PYZ} assumes a slightly different form of the ball condition \eqref{KBallCondition}, where $K=\delta^{-\eps}$. However the case for general $K$ can be obtained by randomly sampling each curve $f\in F$ with probability $K^{-1}$. With high probability, the new family $F'$ will satisfy $\#F'\gtrsim K^{-1}\#F$, and it will satisfy the ball condition \eqref{KBallCondition} with $\delta^{-\eps}$ in place of $K$. \\
{\bf 3.} Theorem 1.7 from \cite{PYZ} applies to a more general class of $C^2$ curves that form a family of \emph{cinematic functions}. For our purposes, it suffices to only consider degree-two polynomials; the set of degree-two polynomials with coefficients in $B(0,R)\subset\RR^3$ is the ``model case'' for a family of cinematic functions. 

\medskip

Theorem \ref{PYZThm} requires that our points satisfy a one-dimensional ball condition.  The next result says that an arbitrary set can be decomposed into pieces that satisfy this condition

\begin{lem}\label{oneDimDecompositionLem}
Let $X\subset\RR^d$ be $\delta$-separated and finite, and let $K\geq 1$. Then we can partition $X = Y\sqcup Z$, and $Z$ can be further decomposed $Z=Z_1\sqcup\ldots Z_N$ so that the following holds. 
\begin{itemize}
\item 
\begin{equation}\label{X1OneDimBall}
\sup_{r\geq\delta,\ x\in \RR^d} \frac{\#\big(B(x,r)\cap Y\big)}{r/\delta}\leq K.
\end{equation}
\item  Each set $Z_i$ is contained in a ball of radius $r_i\gtrsim \delta K^{1/d}$, and each $Z_i$ has a subset $Z_i'$ of size $\#Z_i'\sim r_i/\delta$, with
\begin{equation}\label{X2iOneDimBall}
\sup_{r\geq\delta,\ x\in \RR^d}\frac{|B(x,r)\cap Z_i'|}{r/\delta}\lessapprox 1.
\end{equation}
\end{itemize}
\end{lem}
\begin{proof}
Begin by setting $Y=X$. If \eqref{X1OneDimBall} is satisfied then set $Z = \emptyset$ and we are done. If not, then set $i=1$ and let $B(x_i,r_i)$ be a ball satisfying
\begin{equation}\label{almostMinimalBall}
\frac{|B(x_i,r_i)\cap Y|}{r_i/\delta} \geq \frac{1}{2}\sup_{r\geq\delta,\ x\in \RR^d} \frac{\#\big(B(x,r)\cap Y\big)}{r/\delta}.
\end{equation}
We choose a ball of ``almost minimal'' diameter, in the sense that no ball of radius $\leq r_i/2$ satisfies \eqref{almostMinimalBall}. Define $Z_i=Y\cap B(x_i,r_i)$, and replace $Y$ with $Y\backslash Z_i$. Increment $i$ by one and continue this process until \eqref{X1OneDimBall} is satisfied; this must eventually occur since $X$ is finite and at least one point is removed at each step. Define $Z=\bigsqcup Z_i$.

We will verify that the sets $\{Z_i\}$ satisfy their claimed properties. Since $\frac{\big(B(x_i,r_i)\cap Y\big)}{r_i/\delta}\geq \frac{K}{2}$, we must have $r_i\gtrsim \delta K^{1/d}$. Next, by the minimality of $r_i$, we have
\begin{equation}\label{RiBallWorst}
\sup_{\delta\leq r\leq r_i,\ x\in \RR^d}\frac{\#\big(B(x,r)\cap Z_i\big)}{r/\delta}\lesssim \frac{\#Z_i}{r_i/\delta}.
\end{equation}
But since $Z_i$ is contained in a ball of radius $r_i$, \eqref{RiBallWorst} remains true if the supremum on the LHS is allowed to range over all $r\geq\delta$. 

Let $Z_i'\subset Z_i$ be chosen by randomly selecting each element with probability $\frac{r_i}{\delta(\#Z_i)}$. Then with high probability we have $\#Z_i'\sim r_i/\delta$, and $Z_i'$ satifies \eqref{X2iOneDimBall}.
\end{proof}

Armed with Lemma \ref{oneDimDecompositionLem}, we can now apply Theorem \ref{PYZThm} to understand the RHS of \eqref{volumeBdEachTcube}. The precise statement is as follows. 

\begin{prop}\label{fatterTubesAlmostFullLem}
For all $\eps>0$, there exists $c>0$ so that the following holds. 
Let $\delta\in(0,1]$, let $r\in [4\delta,1]$, and let $t\in [\delta, (\delta r)^{1/2}]$.
Let $(L,Y)_\delta$ and $(L',Y')_\delta$ be sets of $\delta$-separated lines with regular, $\lambda$-dense shadings. Suppose that $L'\subset L\cap B$, for some ball $B = B(\ell_0, t)\subset \mathcal L_{SL_2}$. Suppose that $Y'(\ell')\subset \{x \colon r\leq r_L(x)\leq 2r\}$ for each $\ell'\in L'$. Let $K\geq 1$. Then at least one of the following must happen.
\begin{itemize}
    \item[(A)] There is a line $\tilde \ell\in B(\ell_0,2t)$ and a $\Omega^*(\lambda^4)$-dense shading $\tilde Y$ of $\tilde\ell$ by $t$-cubes, with 
    \begin{equation}\label{wellSpreadCase}
        | Q \cap E_{L}| \geq c\delta^{\eps} K^{-1}  \lambda^6  \Big( (\delta/t) \# L'\Big)|Q|\quad\textrm{for each\ $t$-cube}\ Q\subset \tilde Y(\tilde\ell).
    \end{equation}

    \item[(B)] There exists $\tau\in [\delta K^{1/10}, 1]$ and an arrangement $(\tilde L, \tilde Y)_{\tau}$ of lines with a $\Omega^*(\lambda^4)$-dense shading. For each $\tilde\ell\in\tilde L$ we have
    \begin{equation}\label{clusteredCase}
    | Q \cap E_{L}| \geq c\delta^{\eps}\lambda^6 |Q|\quad\textrm{for each\ $\tau$-cube}\ Q\subset \tilde Y(\tilde\ell).
    \end{equation}
    Finally, $\Omega^*(\# L')$ of the lines $\ell'\in L'$ satisfy $d(\ell',\tilde\ell)\leq 4\tau$, for some $\tilde\ell\in\tilde L$. 
\end{itemize}

\end{prop}
\begin{rem}
Observe that Conclusion (A) is strong when $K$ is small (i.e.~close to 1), while Conclusion (B) is strong when $K$ is large. In practice, we will apply Proposition \ref{fatterTubesAlmostFullLem} for $K=\delta^{-\eps_1}$, where $\eps_1>0$ is very small compared to the quantity $\eps$ from the statement of Theorem \ref{mainThm}. 
\end{rem}
\begin{proof}
Since $L'$ is $\delta$-separated and finite, we can apply Lemma \ref{oneDimDecompositionLem} to $L'$ with parameter $K$; we obtain a partition $L' = L^* \sqcup L^{**}$, and $L^{**}$ can be further partitioned into $L^{**}_1 \sqcup \ldots \sqcup L^{**}_M$. 

\medskip

\noindent {\bf Case (A): $\#L^* \geq \frac{1}{2}\#L'$}.\\
For each $\ell\in L^*$, $\pi_{\ell_0}(\ell)$ is the graph of a polynomial $Ax^2+Bx+C$; identify $\ell$ with the point $(A,B,C)$, and let $F\subset\RR^3$ be the set of all such points, as $\ell$ ranges over the elements of $L^*$. By Lemma \ref{preserveSpacing}, $F$ satisfies the one-dimensional ball condition \eqref{KBallCondition}, with $K_0K=O(K)$ in place of $K$. 

For each $\ell\in L^*$, define $Y^*(\ell)=Y'(\ell)$. Applying Lemma \ref{fiberAboveProjectionIsLarge} with $L$ as above and $(L^*,Y^*)$ in place of $(L',Y')$, we obtain a line $\tilde\ell\in B(\ell_0,2t)$ and a $\Omega^*(\lambda^4)$-dense shading $\tilde Y(\tilde\ell)$ by $t$-cubes. After replacing $(L^*,Y^*)$ with a $\Omega^*(1)$-density refinement, we have that for each $t$-cube $Q\subset \tilde Y(\tilde\ell)$, we have
\begin{equation}\label{ElcapQComputation}
\begin{split}
|E_L\cap Q| & \gtrapprox \lambda^3  t^2 \Big| \bigcup_{\ell \in L^*} \pi_{\ell_0}(Y^{*}(\ell))\Big|\\
& \gtrsim \Big(\lambda^3  t^2\Big)\Big(c_\eps   \delta^{\eps/2}  \big(|\log\delta|^{-2}\lambda\big)^3 K^{-1} (\delta\#L^{*})\Big)\\
&\gtrapprox c_\eps \delta^{\eps/2} K^{-1}\lambda^6\Big((\delta/t)\#L'\Big)|Q|,
\end{split}
\end{equation}
where the second inequality used Theorem \ref{PYZThm}. Selecting $c>0$ sufficiently small depending on $\eps,$ $c_\eps$, and the implicit power of $|\log\delta|$ in \eqref{ElcapQComputation}, we have that \eqref{ElcapQComputation} implies  \eqref{wellSpreadCase}, and Conclusion (A) holds.

\medskip

\noindent {\bf Case (B): $\# L^{**}\geq \frac{1}{2}\#L'$}. \\
After re-indexing, we can select $M'\leq M$ and $\tau\in [\delta K^{1/10}, t]$ so that each of the sets $L^{**}_1\sqcup \ldots \sqcup L^{**}_{M'}$ are contained in balls of diameter between $\tau/2$ and $\tau$, and $\sum_{i=1}^{M'}\# L^{**}_i \gtrapprox \# L'$. Note that $\tau\leq t$, since $L'$ is contained in a ball of radius $t$. For each index $i=1,\ldots,M'$, let $L^{***}_i \subset L^{**}_i$ be a set with $\#L^{***}_i\gtrsim \tau/\delta$ that satisfies the ball condition 
\begin{equation}\label{ballConditionForTubes2i}
\#\big(L^{***}_i\cap B(x,s)\big)\lesssim (s/\delta)\quad\textrm{for all balls}\ B(x,s),
\end{equation}
and $L^{***}_i\subset B(\tilde\ell_i,2\tau)$, for some $\tilde\ell_i\in \mathcal{L}_{SL_2}$. 


Our next task is to construct the set $\tilde L$, the shading $\tilde Y$, and to verify that \eqref{clusteredCase} is satisfied. For each index $i$, we define $F_i\subset\RR^3$ to be the set of all points of the form $(A,B,C),$ where the graph of the polynomial $Ax^2+Bx+C$ is the curve $\pi_{\ell_i}(\ell)$, for some $\ell\in L^{***}_i$. The set $F_i$ satisfies the one-dimensional ball condition \eqref{KBallCondition}, with $O^*(1)$ in place of $K$. 

For each index $i=1,\ldots,M'$, apply Lemma \ref{fiberAboveProjectionIsLarge} with $L$ as above; $L_i^{***}$ in place of $L'$; $\delta$ and $r$ as above; and $\tau$ in place of $t$ (since $\delta\leq\tau\leq t$ and $t\in [\delta,(\delta r)^{1/2}]$ then the same is true for $\tau$). We obtain a line $\tilde\ell_i\in B(\ell_i,2\tau)$ and a $\Omega^*(\lambda^4 )$-dense shading $\tilde Y$ of $\tilde\ell_i$ by $\tau$-cubes. Define $\tilde L = \{\tilde\ell_i\colon i=1,\ldots,M'\}$; this is our desired set of lines $(\tilde L, \tilde Y)_\tau$. By construction, we have $d(\ell', \tilde \ell)\leq 4\tau$ for a $\Omega^*(1)$ fraction of the lines $\ell'\in L'$.

All that remains is to verify \eqref{clusteredCase}. By Lemma \ref{fiberAboveProjectionIsLarge}, after replacing $(L^{***}_i, Y')$ by a $\Omega^*(1)$ refinement, for each $\tau$-cube $Q\subset \tilde Y(\tilde\ell_i)$, we have 
\begin{equation}\label{ElcapQComputationAnalogue}
\begin{split}
|E_L\cap Q| & \gtrapprox \lambda^3  \tau^2 \Big| \bigcup_{\ell \in L_i^{***}} \pi_{\ell_i}(Y'(\ell))\Big|\\
& \gtrapprox \Big(\lambda^3 \tau^2\Big)\Big(c_\eps   \delta^{\eps/2}  \lambda^3  \big(\delta\#L^{***}_i\big)\Big)\\
&\gtrapprox c_\eps \delta^{\eps/2}\lambda^6\tau^2\big(\delta\cdot (\tau/\delta)\big)\\
&\gtrsim c_\eps\delta^{\eps/2}\lambda^6 |Q|.
\end{split}
\end{equation}
If we select $c$ sufficiently small depending on $\eps$, $c_\eps$, and the implicit power of $|\log\delta|$ in the above quasi-inequality, then \eqref{ElcapQComputationAnalogue} implies \eqref{clusteredCase}.
\end{proof}

\section{Proof of Theorem \ref{mainThm}}\label{inductionSection}

In this section we will combine Proposition \ref{fatterTubesAlmostFullLem} with induction on scales to prove Theorem \ref{mainThm}. Our induction step will go from scale $\delta$ to some scale $\tilde\delta\in [\delta^{1-\eps^4},\delta^{1/2}]$. In particular after at most $\eps^{-4}\log(1/\eps)$ iterations we will be at scale $\delta^{\eps}$, at which point we are done. At each step, we will need to consider our set of lines $L$ at both scale $\delta$ and $\tilde\delta$. At the latter scale, our lines might fail to obey the two-dimensional ball condition \eqref{ballCondition}. The next lemma helps us address this problem. 

\begin{lem}\label{nonConcentrationScaleLem}
Let $0<\delta\leq\rho\leq r\leq 1$ and let $L,\tilde L\subset\mathcal{L}_{SL_2}$. 
\begin{itemize}
    \item[(i)] Suppose that $\tilde L$ is $\rho$-separated and contained in a ball of radius $r$. Then there exists a set $\tilde L'\subset\tilde L$ with $\#\tilde L'\gtrsim (\rho/r)(\#\tilde L)$ that satisfies the two-dimensional ball condition $\#(\tilde L' \cap B(\ell,s)) \leq (s/\rho)^2$.

    \item[(ii)] Suppose that $L$ obeys the two-dimensional ball condition $L\cap B(\ell, s)\leq (s/\delta)^2$ and $L\subset \bigcup_{\tilde\ell\in \tilde L}B(\tilde\ell, \rho)$. Then for all $\eps>0$, there exists $c=c(\eps)>0$ and a set $\tilde L'\subset\tilde L$ with $\#\tilde L'\gtrsim c \delta^\eps (\delta/\rho)^2(\# L)$ that satisfies the two-dimensional ball condition $\#(\tilde L' \cap B(\ell,s)) \leq (s/\rho)^2$.
\end{itemize}
\end{lem}
\begin{proof}
For the first estimate, we have $\tilde L\subset B(\ell_0, r)\subset\mathcal{L}_{SL_2}$. We divide $B(\ell_0, r)$ into $O(1)$ pieces $S_1,\ldots,S_k$, and on each piece $S_i$ we select a projection $V_i\colon\RR^4\to\RR^2$ so that each fiber of $V_i$ intersects $S_i$ in a line segment of length $\lesssim r$. This means that on each piece $S_i$, we can refine $\tilde L$ by a factor of $\rho/r$ to obtain a set $\tilde L'$ whose image under $V_i$ is $\rho$-separated, and hence obeys a two-dimensional ball condition. Since $V$ is 1-Lipschitz, we conclude that $\tilde L'$ also obeys a two-dimensional ball condition. 

For the second estimate, by Kaufman's projection theorem we can find a projection $V\colon \RR^4\to\RR^2$ so that $V(L)$ contains a $\delta$-separated set of size $\geq c\delta^{\eps}\#L$, and hence a $\rho$-sepa\-rated set of size $\geq c\delta^{\eps}(\delta/\rho)^2\#L$. Then $V(\tilde L)$ contains a $\rho$-separated set of comparable cardinality; this will be our set $\tilde L'$. Again, since $V(\tilde L')$ satisfies a two-dimensional ball condition and $V$ is 1-Lipschitz, we conclude that $\tilde L'$ also satisfies a two-dimensional ball condition. 
\end{proof}

The following result is a variant of Theorem \ref{mainThm} that has been modified to work well with induction on scales.

\begin{prop}\label{mainProp}
For all $\eps>0$, there exists $c_1=c_1(\eps)>0$ and $c_2=c_2(\eps)>0$ so that the following holds for all $0<\delta\leq\rho \leq 1$ and all $\lambda\in (0,1]$. Let $(L,Y)_\rho$ be an arrangement of lines with a $\lambda$-dense shading. Suppose that $L$ satisfies the two-dimensional ball condition 
\begin{equation}\label{ballConditionInProp}
\#(L \cap B) \leq (r/\rho)^2\quad\textrm{for all}\ r\in[\rho,1],\ \textrm{and all balls}\ B\ \textrm{of radius}\ r.
\end{equation}
Then
\begin{equation}\label{mainPropBigVolumeEstimate}
|E_L| \geq c_1 \rho^\eps \delta^{2\eps} \lambda^W (\rho^2\# L)^{1-c_2},
\end{equation}
where $W=W(\eps,\delta,\rho) = \exp\big(  \frac{100}{\eps^3} \frac{\log \rho}{ \log\delta}\big)$.
\end{prop} 

\begin{rem}\label{boundOnWRemark}
The function $W$ was chosen to have the following three properties.
\begin{itemize}
    \item For $0<\delta\leq\rho\leq 1$, $W(\eps, \delta,\rho)\leq \exp\big(\frac{100}{\eps^3})$; this bound is independent of $\delta$ and $\rho$.
    \item For $\eps,\delta\in (0,1)$ fixed, $W(\eps,\delta,\rho)$ is monotone increasing for $\rho\in (0,1)$.
    \item If $\rho,\delta,\tau\in (0,1)$ and $\tau\geq \rho^{1-\eps^3/10}$, then $W(\eps,\delta, \tau) \leq \exp\big(  \frac{100}{\eps^3}\frac{\log \rho^{1-\eps^3/10} }{ \log\delta}\big) \leq \frac{1}{10} W(\eps, \delta,\rho)$.
\end{itemize}
The first property of $W$ is needed to ensure that Proposition \ref{mainProp} implies Theorem \ref{mainThm}. Indeed, by selecting $\rho=\delta$; replacing $\eps$ by $\eps/4$; and selecting $\delta_0(\eps,R)>0$ sufficiently small, we have that Proposition \ref{mainProp} implies Theorem \ref{mainThm} with $M(\eps) = \exp(100\eps^{-3})$.

Proposition \ref{mainProp} is proved using induction on scales; the second and third properties of $W$ are needed to ensure that the induction closes. 
\end{rem}
\begin{proof}[Proof of Proposition \ref{mainProp}]$\phantom{1}$\\

\noindent {\bf Step 1: Setting up the induction}\\
If $L$ is empty then there is nothing to prove. If $L$ is non-empty, then we can bound the volume of $E_L = \bigcup Y(\ell)$ by the volume of one tube. This gives us the trivial estimate
\begin{equation}\label{trivialEstimate}
|E_L| \geq \lambda\rho^2.
\end{equation}
In particular, if $\delta\geq c_1^{1/\eps}$ then \eqref{trivialEstimate} implies \eqref{mainPropBigVolumeEstimate}. Thus by choosing $c_1(\eps)>0$ appropriately, we may assume that $\delta>0$ is very small compared to $\eps$. For $\delta>0$ fixed, we will prove the result by induction on $\rho$. If $\rho\geq \delta^{\eps}$ and $L$ is non-empty, then \eqref{mainPropBigVolumeEstimate} follows from \eqref{trivialEstimate}. 

In particular, by selecting $c_1$ appropriately, we can suppose that $\rho>0$ is small enough that
\begin{equation}\label{powersOfLogDominatedExponent}
|\log\rho|^{-J}\geq \rho^{1/J},\quad J = \max\Big\{ c_2^{-2},\ \exp\big(\frac{100}{\eps^3}\big)\Big\}.
\end{equation} 
(In practice, we can choose $c_2(\eps)= \eps^{6}$, so there is no circular dependency between $c_1$ and $c_2$). Since $\delta\leq\rho$, the same inequality holds with $\delta$ in place of $\rho$.  A consequence of \eqref{powersOfLogDominatedExponent} is that expressions such as $|\log\rho|^{-W(\eps, \delta,\rho)}$ are bounded below by $\rho^{c_2^2}$. 

For $\delta$ and $\rho$ fixed, we will prove Proposition \ref{mainProp} by induction on $\#L$. If $\#L\leq\rho^{-\eps/2}$ and $L$ is non-empty, then \eqref{mainPropBigVolumeEstimate} follows from the trivial estimate \eqref{trivialEstimate}. Henceforth we will suppose that $\#L>\rho^{-\eps/2}$.

\medskip
\noindent{\bf Step 2: Finding a typical angle $r$.}\\
Apply Lemma \ref{lemRegularShading} to each $\ell\in L$, and let $Y_1(\ell)\subset Y(\ell)$ be a $\lambda/2$-dense regular shading. Define $L_1 = L$. 
Next, by dyadic pigeonholing there exists $r\in [\rho,1]$ so that if we define 
\[
Y_2(\ell) = \{x \in Y_1(\ell)\colon r \leq r_{L_1}(x) \leq 2r\},
\] 
(here $r_L(x)$ is given by Definition \ref{defnRx}), then 
\[
\sum_{\ell\in L}|Y_2(\ell)|\geq|\log\rho|^{-1}\lambda\rho^2\#L.
\] 
Fix this choice of $r$, and define 
\[
L_2 = \big\{\ell \in L_1\colon |Y_2(\ell)|\geq \frac{1}{2}|\log\rho|^{-1}\lambda \rho^2\big\}.
\]
We have
\begin{equation}\label{cardBdL2}
\#L_2 \geq \frac{1}{4}|\log\rho|^{-1}\#L.
\end{equation}
Since $\#L_2(x) \lesssim (r/\rho)^2$ at each point $x\in\RR^3$, we have 
\begin{equation}\label{estimateVerySmallR}
\Big| \bigcup_{\ell \in L_2}Y_2(\ell) \Big| \gtrsim (\rho/r)^2 \sum_{\ell\in L_2}|Y_2(\ell)|\gtrsim (\rho/r)^2 |\log\rho|^{-3}\lambda (\rho^2 \#L).
\end{equation}

In particular, using \eqref{powersOfLogDominatedExponent} to simplify our expression, we conclude that either
\begin{equation}\label{lowerBoundOnR}
r\geq \rho^{1-\eps},
\end{equation}
or else \eqref{mainPropBigVolumeEstimate} follows from \eqref{estimateVerySmallR} and we are done. We will assume henceforth that \eqref{lowerBoundOnR} holds. 

Apply Lemma \ref{lemRegularShading} to each $\ell\in L_2$, and let $Y_3(\ell)\subset Y_2(\ell)$ be a $\frac{1}{4}|\log\rho|^{-1}\lambda$-dense regular shading. Define $L_3 = L_2$. 

\medskip

\noindent{\bf Step 3: A large subset $L_4\subset L_3$ is in a single $1\times r$-tube.}\\
Let $\mathcal{B}$ be a set of $O(r^{-3})$ boundedly overlapping balls of radius $r$ whose union covers $\mathcal{L}_{SL_2}$. For each $B\in\mathcal{B}$, let $L_3^B= L_3\cap B$. 
For each point $x\in E_{L_3}$ we have $r_{L_3}(x)\in [r,2r]$, and hence there are $O(1)$ balls $B\in\mathcal{B}$ with $x\in E_{L_3^B}$. Thus 
\begin{equation}\label{boundedOverlap}
\Big| \bigcup_{\ell \in L_3}Y_3(\ell) \Big| \gtrsim \sum_{B\in\mathcal{B}} \Big| \bigcup_{\ell \in L_3^B}Y_3(\ell) \Big|. 
\end{equation}
Let $N = \sup_{B\in\mathcal{B}} \#L_3^B$. By our induction hypothesis we have
\begin{equation}\label{gainFromSmallN}
\begin{split}
\sum_{B\in\mathcal{B}} \Big| \bigcup_{\ell \in L_3^B}Y_3(\ell) \Big| & \geq c_1\rho^\eps \delta^{2\eps} \big( \frac{1}{4}|\log\rho|^{-1}\lambda\big)^{W(\eps,\delta,\rho)}\sum_{B\in\mathcal{B}}(\rho^2\#L_3^B)^{1-c_2}\\
& \geq c_1\rho^\eps \delta^{2\eps}\lambda^{W(\eps,\delta,\rho)} \big( \frac{1}{2}|\log\rho|^{-1}\big)^{W(\eps,\delta,\rho)}(\rho^2\#L_3)^{1-c_2} \big(\frac{\#L_3}{N}\big)^{c_2}\\
&\geq c_1\rho^\eps \delta^{2\eps}\lambda^{W(\eps,\delta,\rho)} (\rho^2\# L)^{1-c_2} \Big( (2|\log\rho|)^{-W(\eps,\delta,\rho)-1} \big(\frac{\#L}{N}\big)^{c_2} \Big),
\end{split}
\end{equation}
where the final line used \eqref{cardBdL2}. Using \eqref{powersOfLogDominatedExponent} to analyze the final term in brackets, we conclude that either there is a set $L_4$ of the form $L_4= L_3^{B_0}$ with
\begin{equation}\label{lowerBoundOnN}
N = \#L_4 \geq \rho^{c_2}(\#L),
\end{equation}
or else \eqref{mainPropBigVolumeEstimate} follows from \eqref{boundedOverlap} and \eqref{gainFromSmallN} and we are done. 
We will assume henceforth that there is a set $L_4 = L_3^{B_0}$ satisfying \eqref{lowerBoundOnN}; define the shading $Y_4(\ell)=Y_3(\ell),\ \ell\in L_4$.

\medskip

\noindent{\bf Step 4: Clustering at scale $t=(\rho r)^{1/2}$.}\\
Recall that $L_4$ is contained in a ball $B_0\subset\RR^4$ of radius $r$. Let $\mathcal{B}_0$ be a covering of $B_0\cap \mathcal{L}_{SL_2}$ by boundedly overlapping balls of radius $t=(\rho r)^{1/2}$. After pigeonholing, we select an integer $M$ and a set $\mathcal{B}\subset \mathcal{B}_0$ so that 
\[
\sum_{B\in\mathcal{B}}\#(L_4\cap B)\geq|\log\rho|^{-1}\#L_4,
\] 
and 
\[
\#(B\cap L_4)\in [M, 2M)\quad\textrm{for all}\ B\in\mathcal{B}.
\] 

Define 
\[
\tilde L = \{\ell_B\colon B\in \mathcal{B}\},
\] 
where $\ell_B\in\mathcal{L}_{SL_2}$ is the center of $B\subset\RR^4$. For each $\tilde \ell \in \tilde L$, define $L_4(\tilde \ell) = L_4\cap B(\tilde\ell, t)$. Abusing notation slightly, we define the pairs $(L_4(\tilde \ell), Y_4)_\rho$, where $Y_4$ is the restriction of the shading $Y_4$ from $L_4$ to $L_4(\tilde \ell)$. We have
\begin{equation}\label{mostLinesInL3}
M(\#\tilde L) \gtrsim |\log\rho|^{-1}(\#L_4), \quad\textrm{and}\quad \#(L_4(\tilde\ell))\in [M,2M]\ \textrm{for all}\ \tilde\ell\in \tilde L. 
\end{equation}

\medskip

\noindent{\bf Step 5: Applying Proposition \ref{fatterTubesAlmostFullLem}.}\\
For each $\tilde\ell \in \tilde L$, Apply Proposition \ref{fatterTubesAlmostFullLem} with 
\begin{itemize}
\item $\rho$ in place of $\delta$; $\frac{1}{4}|\log\rho|^{-1}\lambda$ in place of $\lambda$; $r$ and $t=(\rho r)^{1/2}$ as above. 
\item $(L_1,Y_1)$ in place of $(L,Y)$.
\item $(L_4(\tilde \ell), Y_4)$ in place of $(L',Y')$.
\item $\tilde\ell$ in place of $\ell_0$
\item $K=\rho^{-\eps^3}$ and $\eps_1 = \eps^4/40$ in place of $\eps$. 
\end{itemize}
We will say that $\tilde\ell$ is of \emph{Type (A)} if Conclusion (A) holds, and is of \emph{Type (B)} if Conclusion (B) holds. 

\medskip
\noindent{\bf Case (A)}: If $\tilde\ell$ is of Type (A), then there is a shading $\tilde Y(\tilde \ell)$ that is a union of grid-aligned $t$-cubes, with $|\tilde Y(\tilde\ell)|\gtrapprox \lambda^4 t^2$, so that for each $t$-cube $Q\subset \tilde Y(\tilde\ell)$, we have 
\begin{equation}\label{howFullEachTildeDeltaCubeTypeA}
\begin{split}
|Q \cap E_L| & \geq |Q\cap E_{L_1}| \geq c_{\eps_1}\rho^{\eps_1} K^{-1}\big(\frac{1}{4} |\log\rho|^{-1}\lambda\big)^6   \big( (\rho/t)M \big)|Q|\\
&\geq   c_\eps \rho^{\eps^4/20 + \eps^3} \lambda^6  \big( (\rho/r)^{1/2}M \big)|Q|.
\end{split}
\end{equation}
(In the second line we replaced $c_{\eps_1}$ by $c_\eps$ since $\eps_1$ is a function of $\eps$, and used \eqref{powersOfLogDominatedExponent} to simplify the expression).

\medskip
\noindent{\bf Case (B)}:
If $\tilde\ell$ is of type (B), then there exists a diameter $\tau=\tau(\tilde\ell)\in [\rho^{1-\eps^3/10}, 1]$ and a set $(L^*_{\tilde\ell},Y^*_{\tilde\ell})_{\tau}$ of lines with a $\Omega^*(\frac{1}{4}|\log\rho|^{-1}\lambda)=\Omega^*(\lambda^4)$ dense shading (we use the subscript $\tilde\ell$ to emphasize that the set $L^*$ and the shading $Y^*$ depend on the choice of lines $\tilde\ell$). For each $\ell^*\in L^*_{\tilde\ell}$ and each $\tau$ cube $Q$, we have
\begin{equation}\label{howFullEachTildeDeltaCubeTypeB}
|Q \cap E_L| \geq |Q\cap E_{L_1}|  \geq c_{\eps_1}\rho^{\eps_1} \big(\frac{1}{4} |\log\rho|^{-1}\lambda\big)^6|Q| \geq c_{\eps}\rho^{\eps^4/30}\lambda^6|Q|.
\end{equation}
Finally, $\Omega^*(\#L_4(\tilde \ell))$ of the lines $\ell \in L_4(\tilde \ell)$ satisfy $d(\ell, \ell^*)\leq4\tau$ for some $\ell^*\in L^*_{\tilde\ell}$.

\medskip

\noindent{\bf Step 6: Closing the induction if most lines are of Type (A).}\\
Suppose at least half the lines from $\tilde L$ are of Type (A). Abusing notation slightly, we will refine the set $\tilde L$ to consist only of these lines. Apply Lemma \ref{nonConcentrationScaleLem} with $L_4$ in place of $L$; $\tilde L$ as above; $\rho$ in place of $\delta$; $t$ in place of $\rho$; and $r$ as above. We obtain a set $\tilde L'\subset\tilde L$ that obeys a two-dimensional ball condition and has size
\begin{equation}\label{boundOnCardTildeL}
\#\tilde L' \gtrsim (t/r)(\# \tilde L) \gtrapprox_\rho (\rho/r)^{1/2} N/M,
\end{equation}
where $N$ obeys the bound \eqref{lowerBoundOnN}. 

We can now apply the induction hypothesis at scale  $t$ to $(\tilde L', \tilde Y)$ with $\tilde\lambda = \Omega^*(\lambda^4)$, i.e. $\tilde\lambda\geq c|\log\rho|^{-C_0}\lambda^4$, for some absolute constant $C_0$. We first record the following computation. By \eqref{lowerBoundOnR}, we have $t \geq \rho^{1-\eps},$ and thus by the computation in Remark \ref{boundOnWRemark}, 
\begin{equation}\label{boundOnW}
W(\eps, \delta, t) \leq \frac{1}{10} W(\eps, \delta,\rho).
\end{equation}

Using \eqref{boundOnCardTildeL}, \eqref{boundOnW}, and using \eqref{powersOfLogDominatedExponent} to simplify our expression, our induction hypothesis says the following.
\begin{equation*}
\begin{split}
\Big|\bigcup_{\tilde l\in \tilde L'} \tilde Y(\tilde\ell)\Big| & \geq c_1 \delta^{2\eps}t^\eps \tilde\lambda^{W(\eps, \delta, t)} (t^2\# \tilde L')^{1-c_2}\\
& \gtrapprox_\rho c_1 t^\eps\delta^{2\eps} ( c\lambda^4|\log\rho|^{-C_0})^{\frac{1}{10}W(\eps, \delta,\rho)} \big(t^2  (\rho/r)^{1/2} N/M\big)^{1-c_2}\\
& \geq \Big( \rho^{-\eps^2}(c|\log\rho|^{-C_0})^{\exp(100/\eps^3)}   \Big) c_1 \rho^\eps \delta^{2\eps}\lambda^{\frac{1}{2}W(\eps,\delta,\rho)}  \big(t^2(\rho/r)^{1/2} N/M\big)^{1-c_2}.
\end{split}
\end{equation*}
If $c_1$ is chosen sufficiently small (depending on $\eps$ and $c$), then this forces $\rho>0$ to be sufficiently small that the first term in brackets is bounded below by $\rho^{-\eps^2/2}$, i.e.
\begin{equation}\label{volumeScaleTildeRho}
\Big|\bigcup_{\tilde l\in \tilde L'} \tilde Y(\tilde\ell)\Big|  \geq  \rho^{-\eps^2/2} c_1 \rho^\eps \delta^{2\eps}\lambda^{\frac{1}{2}W(\eps,\delta,\rho)}  \big(t^2(\rho/r)^{1/2} N/M\big)^{1-c_2}.
\end{equation}

Recall that the set on the LHS of \eqref{volumeScaleTildeRho} is a union of $t$ cubes, and \eqref{volumeScaleTildeRho} gives us a lower bound on the volume of the union of these cubes. On the other hand, \eqref{howFullEachTildeDeltaCubeTypeA} gives a lower bound for the size of the intersection of each of these cubes with $E_L$. By combining these inequalities, we obtain a lower bound on $E_L$. The computation is as follows. 
\begin{equation}
\begin{split}
|E_L| & \gtrapprox_\rho
\Big[ c_\eps \rho^{\eps^4/20 + \eps^3} \lambda^6  \big( (\rho/r)^{1/2}M\Big]
\Big[\rho^{-\eps^2/2} c_1 \rho^\eps \delta^{2\eps}\lambda^{\frac{1}{2}W(\eps,\delta,\rho)}  \big(t^2(\rho/r)^{1/2}  N/M\big)^{1-c_2}\Big]\\
& \geq \Big( c_\eps\rho^{\eps_1+3c_2+\eps^3-\eps^2/2}   \Big) \big(c_1 \rho^\eps \delta^{2\eps}\lambda^{W(\eps,\delta,\rho)}\big) \big(\rho^{2} \#L   \big)^{1-c_2},
\end{split}
\end{equation}
where we used \eqref{lowerBoundOnN} to replace $N$ by $\rho^{2c_2}(\#L)$ and $t=(\rho r)^{1/2}$. Since $\eps_1<\eps^4/40$, if we select $c_2 < \eps^2/100$ and if $c_1(\eps)>0$ is sufficiently small, which in turn makes $\rho>0$ sufficiently small, then the first term in brackets is larger than $|\rho\delta|^{-C_1}$, where $C_1$ is the implicit power of $|\log\rho|$ in the above quasi-inequality. Thus the induction closes. This concludes the proof of Proposition \ref{mainProp} if at least half the lines from $\tilde L$ are of Type (A).

\medskip

\noindent{\bf Step 7: Closing the induction if most lines are of Type (B).}\\
Suppose at least half the lines from $\tilde L$ are of Type (B). By pigeonholing, we can select a popular diameter $\tau\in [\rho^{1-\eps^3/10}, 2r]$ and define 
\[
\tilde L_5=\{\tilde\ell\in \tilde L\colon \tau(\tilde\ell)\in [\tau/2,\tau]\},\quad L_5^*=\bigcup_{\tilde\ell\in \tilde L_5}L^*_{\tilde\ell}.
\]
Define the shading $Y_5^*$ on $L_5^*$ in the natural way, i.e.~$Y_5^*(\ell^*) = Y^*_{\tilde\ell}(\ell^*)$, where $\ell^*\in L^*_{\tilde\ell}$. Finally, define 
\[
L_5=\bigcup_{\tilde\ell\in \tilde L_5}\big\{\ell\in L_4(\tilde \ell)\colon d(\ell, \ell^*)\leq \tau \quad\textrm{for some}\ \ell^*\in L^*_{\tilde\ell} \big\}.
\]
We have $\#L_5\gtrapprox_\rho \#L_4.$ Next, apply Lemma \ref{nonConcentrationScaleLem} with $L_5$ in place of $L$; $L_5^*$ in place of $\tilde L$; $\rho$ in place of $\delta$; $\tau$ in pace of $\tilde\delta$; $r$ as above; and $\eps_1 = \eps^4/40$ in place of $\eps$. We can verify that the hypotheses of Part (ii) of the lemma are satisfied: since $L$ obeys a two-dimensional ball condition, so does $L_5$. Furthermore $L_5\subset\bigcup_{\ell^*\in L_5^*}B(\ell^*,\rho^*)$ by construction. 

We obtain a sets $L_6$ and $L_6^*$, so that $\#L_6^*$ satisfies a two-dimensional ball condition, and
\begin{equation}
\#L_6^* \geq c'_\eps \rho^{\eps_1} (\rho/\tau)^2(\# L_6) \geq c'_\eps\rho^{\eps^4/20}(\rho/\tau)^2\#L,
\end{equation}
where the second inequality made use of the fact that each set $L\supset L_1\supset\ldots\supset L_6$ has comparable size, up to harmless $O^*(1)$ factors that are accounted for by replacing $\rho^{\eps_1}$ with $\rho^{\eps^4/20}$ Let $Y_6^*$ be the restriction of $Y_5^*$ to $L_6^*$.

We will apply the induction hypothesis at scale $\tau$ to $(L_6^*, Y_6^*)_{\tau}$. Since $\tau\geq \rho^{1-\eps^3/10}$, the computation in Remark \ref{boundOnWRemark} shows that $W(\eps,\delta,\tau)\leq\frac{1}{10}W(\eps,\delta,\rho)$. We can now compute
\begin{equation*}
\begin{split}
\Big|\bigcup_{\ell^*\in L_6^*}Y_6^*(\ell^*) \Big|&\geq c_1 [\tau^\eps] \delta^{2\eps}\Big(\frac{1}{4}|\log\rho|^{-1}\lambda\Big)^{W(\eps, \delta, \tau)}\Big({\tau}^2\#L_6^*\Big)^{1-c_2}\\
&\gtrapprox c_1 [\rho^{\eps}\rho^{-\eps^4/10}] \delta^{2\eps}\Big(\frac{1}{4}|\log\rho|^{-1}\lambda\Big)^{W(\eps, \delta, \tau)}\Big( c'_\eps\rho^{\eps^4/20}(\rho^2\#L)\Big)^{1-c_2}\\
&\geq\Big(c'_\eps \rho^{-\eps^4/10+\eps^4/20}|\log\rho|^{-C_0W(\eps,\delta,\rho)} \Big)\Big( c_1\rho^\eps\delta^{2\eps}\lambda^{\frac{1}{2}W(\eps,\delta,\rho)}(\rho^2\#L)^{1-c_2}\Big).
\end{split}
\end{equation*}
Observe that the first term in brackets becomes larger as $\rho\searrow 0$. Selecting $c_1$ sufficiently small, which in turn forces $\rho$ too be sufficiently small, we obtain
\begin{equation}\label{volumeScaleT}
\Big|\bigcup_{\ell^*\in  L_6^*} Y_6^*(\ell^*) \Big|\geq \rho^{-\eps^4/25}\Big( c_1\rho^\eps\delta^{2\eps}\lambda^{\frac{1}{2}W(\eps,\delta,\rho)}(\rho^2\#L)^{1-c_2}\Big).
\end{equation}
The set on the LHS of \eqref{volumeScaleT} is a union of $\tau$-cubes, and \eqref{volumeScaleT} gives us a lower bound on the volume of the union of these cubes. On the other hand, \eqref{howFullEachTildeDeltaCubeTypeB} gives a lower bound for the size of the intersection of each of these cubes with $E_L$. Thus by combining these inequalities, we obtain a lower bound on $E_L$. The computation is as follows. 

\begin{equation}
\begin{split}
|E_L| & \geq
\Big[c_{\eps}\rho^{\eps^4/30}\lambda^6 \Big]
\Big[\rho^{-\eps^4/25}  c_1\rho^\eps\delta^{2\eps}\lambda^{\frac{1}{2}W(\eps,\delta,\rho)}(\rho^2\#L)^{1-c_2} \Big]\\
& = \Big(c_\eps\rho^{\eps^4/30 -\eps^4/25} \Big)\Big(c_1\rho^\eps\delta^{2\eps}\lambda^{W(\eps,\delta,\rho)}(\rho^2\#L)^{1-c_2} \Big).
\end{split}
\end{equation}
If $c_1(\eps)>0$ is sufficiently small, which in turn makes $\rho>0$ sufficiently small, then the first term in brackets is larger than 1 and the induction closes. This concludes the proof of Proposition \ref{mainProp} if at least half the lines from $\tilde L$ are of Type (B).
\end{proof}

\section{Reversing the F\"assler-Orponen argument}\label{reversingFOSec}
In this section, we will briefly summarize some of the key steps in F\"assler and Orponen's proof of Conjecture \ref{sl2Conj} from \cite{FO}, and explain how this proof can be reversed to show that Theorem \ref{mainThm} implies Theorem \ref{GGGHMWThm}. The main steps from the F\"assler-Orponen argument are shown in Figure \ref{FOPic}.

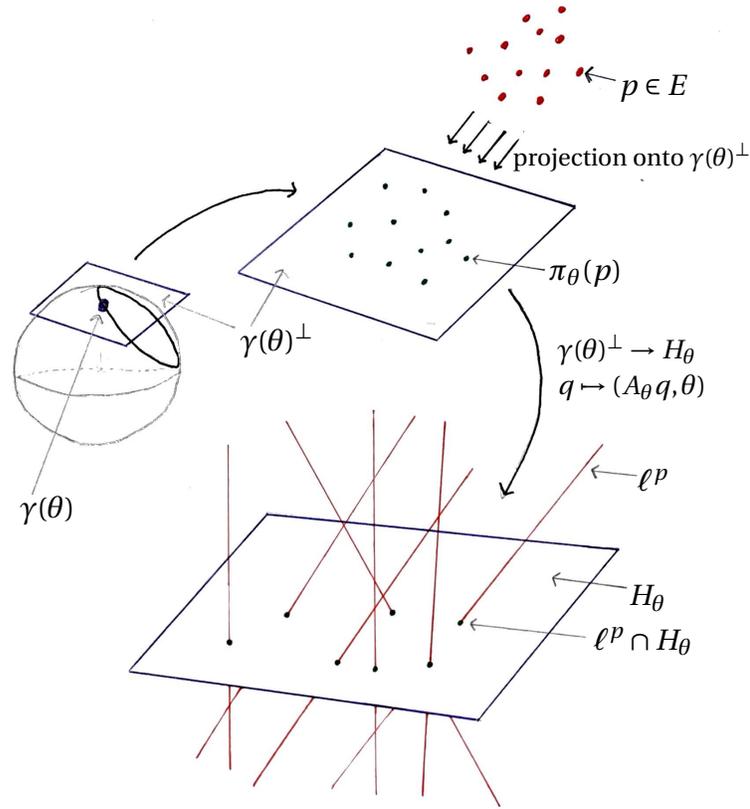
\begin{figure}[h]
\centering
\begin{overpic}[ scale=0.5]{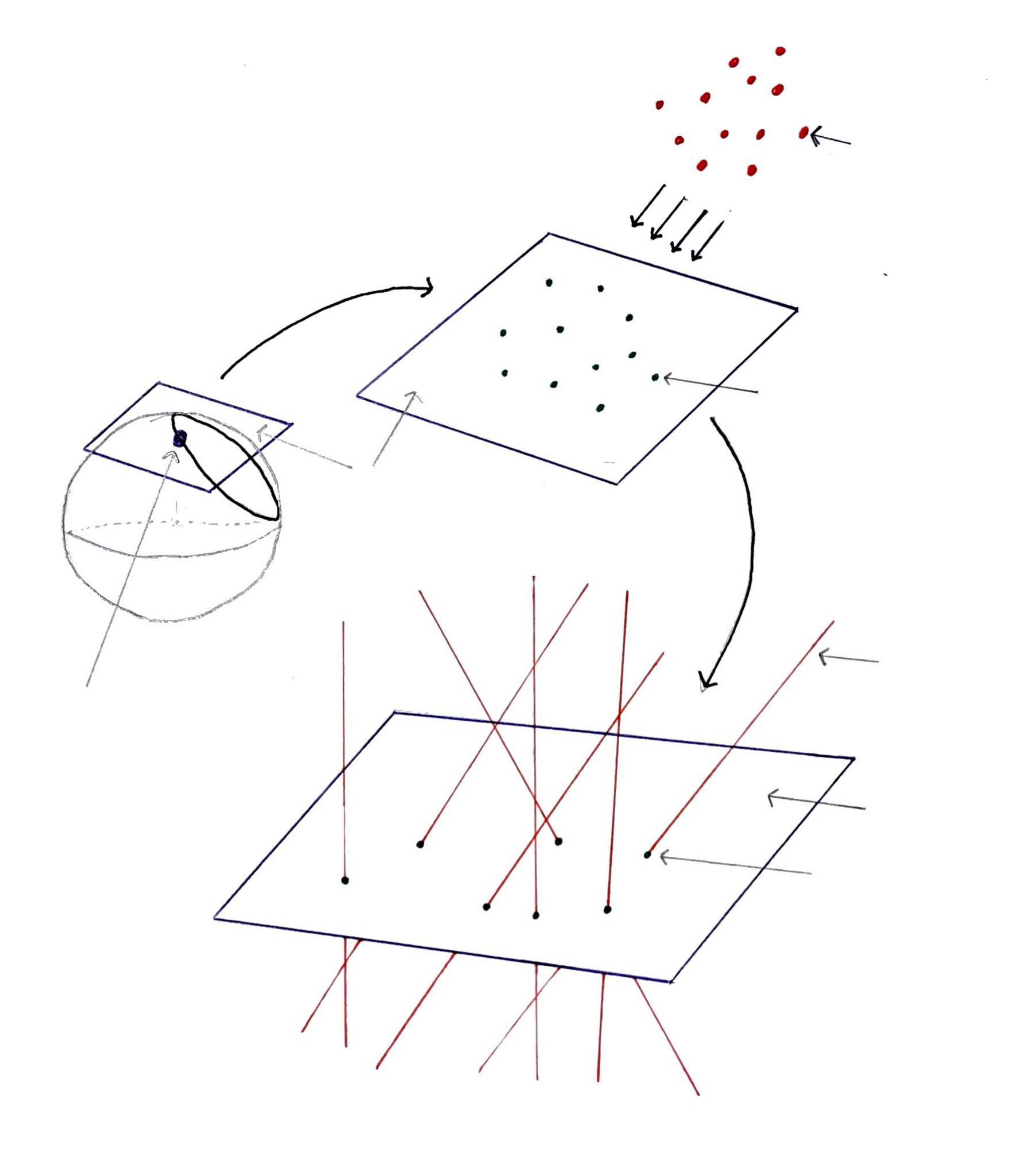}
 \put (6,38) {$\gamma(\theta)$}
 \put (31,57) {$\gamma(\theta)^\perp$}
 \put (74,86) {$p\in E$}
 \put (62,78) {{\footnotesize projection onto $\gamma(\theta)^\perp$}}
 \put (66,65) {$\pi_\theta(p)$}
 \put (67,56) {{\small $\gamma(\theta)^\perp\to H_\theta$} }
 \put (67,52) {{\small $q\mapsto (A_\theta q, \theta)$} }
 \put (76,41) {$\ell^p$}
 \put (75,28) {$H_\theta$}
 \put (71,23) {$\ell^p\cap H_\theta$}
\end{overpic}
\caption{The F\"assler-Orponen reduction}
\label{FOPic}
\end{figure}

We will briefly expand on a few items from Figure \ref{FOPic}. Define
\begin{equation}\label{defnGammaTheta}
\gamma(\theta) = \frac{(1,-\theta,\ \theta^2/2)}{\Vert (1,-\theta,\ \theta^2/2)\Vert},\quad \pi_\theta(p) = p- \big(p\cdot\gamma(\theta)\big)\gamma(\theta).
\end{equation}
The curve, $\gamma\colon [0,1]\to S^2$ parameterizes part of the right-angled light cone pointing in the direction $\big(\frac{1}{\sqrt 2}, 0, \frac{1}{\sqrt 2}\big)$ (note that this is not a unit speed parameterization).  

For each $\theta\in [0,1],$ let $A_\theta\colon \gamma(\theta)^\perp\to\RR^2$ be given by 
\[
A_\theta(q) = \big(q\cdot (\theta, 1, 0),\ q\cdot(0, \theta/2, 1)\big).
\]
Observe that $A_\theta$ and $A_\theta^{-1}$ are $100$-Lipschitz for all $\theta\in [0,1]$. For $p=(p_1,p_2,p_3)\in\RR^3$, define the line (note that $ad-bc=1$)
\begin{equation}
\begin{split}
&\ell^p = \ell_{(a,b,c,d)},\\ 
a = p_3-2p_1p_2/p_3,\quad &b= -2p_2/p_3,\quad c=p_1/p_3,\quad d=1/p_3.
\end{split}
\end{equation}
The map $p\mapsto\ell^p$ is Lipschitz on compact subsets of $\RR^3\backslash\{p_3=0\}$. A key step in the F\"assler-Orponen proof is the observation that 
\begin{equation}\label{AthetaTheta}
(A_\theta(\pi_\theta(p)), \theta) = \ell^p\cap H_\theta,\quad p\in\RR^3\backslash\{p_3=0\},\ \theta\in[0,1].
\end{equation}

We will exploit \eqref{AthetaTheta} as follows. Let $P\subset\RR^3$ and let $G\subset P \times [0,1]$. Define the slices
\[
G_\theta = \{p\in P\colon (p,\theta)\in G\},\quad G^p = \{\theta\in [0,1]\colon (p,\theta)\in G\}.
\]

For each $p\in P$, define $Y(\ell^p)$ to be the union of $\delta$-cubes that intersect the set 
\begin{equation}\label{defineShadingY}
\{ A_\theta(\pi_\theta(p)), \theta) \colon \theta\in G^p \}.
\end{equation}

By \eqref{AthetaTheta}, $Y(\ell^p)$ is a shading of $\ell^p$, i.e. each $\delta$-cube in $Y(\ell^p)$ intersects $\ell^p$. Furthermore, the shading $Y(\ell^p)$ has density $\Omega(\mathcal{E}_{\delta}( G^p))$, where $\mathcal{E}_{\delta}(X)$ denotes the $\delta$ covering number of the set $X$.  With these definitions, we can state the main result of this section. In what follows, $\gamma(\theta)$ and $\pi_\theta$ are as specified in \eqref{defnGammaTheta}.

\begin{prop}\label{projMomentCurve}
For all $\eps>0$, $R\geq 1$, there exists $M=M(\eps)$ and $\delta_0=\delta_0(\eps,R)$ so that the following holds for all $\delta\in(0,\delta_0]$. Let $P\subset \{p\in \RR^3\colon |p|\leq R, |p_3|\geq 1/R\}$, and suppose $P$ satisfies the ball condition 
\begin{equation}\label{ballConditionForP}
\mathcal{E}_\delta(P \cap B) \leq (r/\delta)^2\quad\textrm{for all}\ r\in[\delta,R],\ \textrm{and all balls}\ B\ \textrm{of radius}\ r.
\end{equation}

Let $G\subset P\times[0,1]$, with $\mathcal{E}_\delta(G) \geq \lambda \delta^{-1}\mathcal{E}_\delta(P)$. Then there exists $\theta\in[0,1]$ so that
\begin{equation}\label{largeProjectionTheta}
\mathcal{E}_{\delta} (\pi_\theta(G_\theta))\geq \delta^{\eps}\lambda^M \mathcal{E}_{\delta}(P).
\end{equation}
\end{prop}
Proposition \ref{projMomentCurve} implies Theorem \ref{GGGHMWThm} by standard discretization arguments. We will remark briefly on how Proposition \ref{projMomentCurve} follows from Theorem \ref{mainThm}. First, we may suppose that $G$ is a union of grid-aligned $\delta$-cubes. Let $L$ be a maximal $\delta$-separated subset of $\{\ell^p\colon p\in P\}$, so \eqref{ballConditionForP} implies that $L\subset \mathcal{L}_{SL_2}$ satisfies \eqref{ballCondition}. For each $\ell=\ell^p\in L$, we define the shading $Y(\ell)$ as above (see \eqref{defineShadingY}). After a harmless refinement by a factor of $O^*(1)$, we may suppose that each shading $Y(\ell)$ has density $\Omega^*(\lambda)$. By Theorem  \ref{mainThm} and Fubini, there is a choice of $\theta\in [0,1]$ so that
\[
\Big| \{z=\theta\}\cap \bigcup_{\ell\in L}Y(\ell) \Big| \geq \delta^{\eps}\lambda^M(\delta^2\# L),
\]
but this is precisely \eqref{largeProjectionTheta}.


\end{document}